\documentclass[11pt]{article}
\usepackage{amsmath}
\usepackage{amsfonts}
\usepackage{amsthm}
\usepackage{amssymb, setspace}
\usepackage{color,graphicx,epsfig,geometry,hyperref,fancyhdr}
\usepackage[T1]{fontenc}
\usepackage{float}
\usepackage{subfig}
\usepackage{commath}
\usepackage{bbm}
\usepackage{mathrsfs,fleqn}
\usepackage{pgfplots}
\pgfplotsset{compat=newest}
\newtheorem{theorem}{Theorem}[section]

\newtheorem{example}{Example}[section]
\newcommand{\N}{\mathbb{N}}
\newcommand{\Z}{\mathbb{Z}}
\newcommand{\weakc}{\rightharpoonup}
\newcommand{\R}{\mathbb{R}}
\newcommand{\C}{\mathbb{C}}

\newcommand{\grad}{\nabla}

\usepackage{ulem}

\graphicspath{{paper-pics/}}

\begin{document}

\begin{flushleft}
\Large 
\noindent{\bf \Large Analysis of the transmission eigenvalue problem with two conductivity parameters}
\end{flushleft}

\vspace{0.2in}
{\bf  \large Rafael Ceja Ayala and Isaac Harris}\\
\indent {\small Department of Mathematics, Purdue University, West Lafayette, IN 47907 }\\
\indent {\small Email:  \texttt{rcejaaya@purdue.edu}  and \texttt{harri814@purdue.edu} }\\

{\bf  \large Andreas Kleefeld}\\
\indent {\small Forschungszentrum J\"{u}lich GmbH, J\"{u}lich Supercomputing Centre, } \\
\indent {\small Wilhelm-Johnen-Stra{\ss}e, 52425 J\"{u}lich, Germany}\\
\indent {\small University of Applied Sciences Aachen, Faculty of Medical Engineering and } \\
\indent {\small Technomathematics, Heinrich-Mu\ss{}mann-Str. 1, 52428 J\"{u}lich, Germany}\\
\indent {\small Email: \texttt{a.kleefeld@fz-juelich.de}}\\

{\bf  \large Nikolaos Pallikarakis}\\
\indent {\small Department of Mathematics, National Technical University of Athens,}\\
\indent {\small 15780 Athens, Greece}\\
\indent {\small Email:  \texttt{npall@central.ntua.gr} }\\

\begin{abstract}
\noindent In this paper, we provide an analytical study of the transmission eigenvalue problem with two conductivity parameters. We will assume that the underlying physical model is given by the scattering of a plane wave for an isotropic scatterer. In previous studies, this eigenvalue problem was analyzed with one conductive boundary parameter where as we will consider the case of two parameters. We will prove the existence and discreteness of the transmission eigenvalues as well as study the dependance on the physical parameters. We are able to prove monotonicity of the first transmission eigenvalue with respect to the parameters and consider the limiting procedure as the second  boundary parameter vanishes. Lastly, we provide extensive numerical experiments to validate the theoretical work.
\end{abstract}

\section{Introduction}
In {this} paper, we will study the transmission eigenvalue problem for an acoustic isotropic scatterer with two conductive boundary conditions. Transmission eigenvalues have been a very active field of investigation in the area of inverse scattering. This is due to the fact that these eigenvalues can be recovered from the far-field data see for e.g. \cite{far field data,anisotropic} as well as can be used to determine defects in a material \cite{TA,cavities,electro,mypaper1,te-cbc3}.  In general, one can prove that the  transmission eigenvalues depend monotonically on the physical parameters, which implies that they can be used as a target signature for non-destructive testing. Non-destructive testing arises in many applications such as engineering and medical imaging, i.e. one wishes to recover information about the interior structure given exterior measurements. Therefore, by having information or knowledge of the transmission eigenvalues, one can retrieve information about the material properties of the scattering object. Another reason one studies these eigenvalue problems, {is their non-linear and non-self-adjoint nature}. This makes them mathematically challenging to study. 

Deriving accurate numerical algorithms to compute the transmission eigenvalues is an active field of study see for e.g. \cite{spectraltev1,spectraltev2,fem-te,CMS,GP,numerical,mfs-te,eig-FEM-book}. As mentioned, here we consider the scalar transmission eigenvalue problem with a two parameter conductive boundary condition denoted $\lambda$ and $\eta$. This problem was first introduced in \cite{fmconductbc}. The eigenvalue problem with one conductive boundary condition has been studied in \cite{te-cbc,te-geo-paper1,te-cbc2,te-cbc3} for the case of acoustic scattering where as in \cite{two-eig-cbc,electro-cbc} for electromagnetic scatterers. Due to the presence of the second parameter in the conductive boundary condition the analysis used in the aforementioned manuscripts will not work for the problem at hand. Therefore, we will need to use different analytical tools to study our transmission eigenvalue problem. 

The rest of the paper is organized as follows. We will derive the transmission eigenvalue problem under consideration from the direct scattering problem in Section \ref{theproblem}. Next, in Section \ref{discrete} we prove that the transmission eigenvalues form a discrete set in the complex plane as well as provide and example via separation of variables to prove that this is a non-selfadjoint eigenvalue problem. Then in Section \ref{exist}, we prove the existence of infinitely many real transmission eigenvalues as well as study the dependance on the material parameters. {Furthermore,} in Section \ref{limit} we consider the limiting process as $\lambda \to 1$ where we are able to prove that the transmission eigenpairs converge to the eigenpairs for one conductive boundary parameter i.e. with $\lambda =1$. Numerical examples, using separation of variables are given in Section \ref{numerics} to validate the analysis presented in the earlier sections. Future, numerical results are given using boundary integral equations.

\section{Formulation of the problem}\label{theproblem}
We will now state the transmission eigenvalue problem under consideration by connecting it to the direct scattering problem. To this end, we will formulate the direct scattering problem associated with the transmission eigenvalues in $\mathbb{R}^d$ where $d=2$ or $d=3$. Let $D\subset \mathbb{R}^d$ be a simply connected open set with $C^2$ boundary $\partial D$ where $\nu$ denotes the unit outward normal vector. We then assume that the refractive index $n \in L^{\infty}(D)$ satisfies  
$$0<n_{{min}} \leq n(x)\leq n_{{max}}< \infty \quad \text{for a.e. } \,\, x \in D.$$ 
We are particularly interested in the case where there exists two (conductivity) boundary parameters $\lambda$ and $\eta$ as in \cite{fmconductbc}. {These parameters occur e.g. when the scattered medium is enclosed by a thin layer with high conductivity \cite{Schm}.} Therefore, we assume $\eta \in L^{\infty}(\partial D)$ such that 
$$ \eta_{{min}} \leq \eta(x)\leq \eta_{{max}} \quad \text{for a.e. } \,\, x \in \partial D$$
and fixed constant $\lambda \neq 1$. 
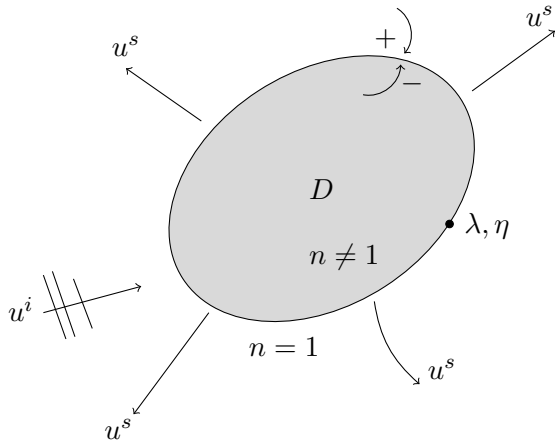
\begin{figure}[ht]
\centering 
\begin{tikzpicture}
\draw[fill=gray!30, rotate=33] (0,0) ellipse (2.2cm and 1.55cm) node [thick] {$D$};
\draw (0.3,-0.9) node {$n\neq1$};
\draw (-0.5,-2.1) node {$n=1$};

\draw[bend right=20, ->] (0.7,-1.5) to (1.3, -2.6);
\draw (1.6,-2.4) node {$u^s$};


\draw[fill] (1.7,-0.47) circle (0.05);

\draw [thick] (2.2,-0.5) node {$\lambda,\eta$};

\draw [->] (2,1.3)--(3.1,2.1);
\draw (2.9,2.3) node {$u^s$};

\draw[bend left=50, ->] (1,2.4) to (1.1, 1.8);
\draw (0.85,2) node {$+$};

\draw[bend right=50, ->] (0.55,1.25) to (1.05, 1.65);
\draw (1.2,1.4) node {$-$};

\draw [->] (-1.6,0.9)--(-2.6,1.6);
\draw (-2.6,1.9) node {$u^s$};

\draw [->] (-1.5,-1.65)--(-2.5,-3);
\draw (-2.7,-3.2) node {$u^s$};

\draw [->] (-3.7,-1.65)--(-2.4,-1.3);
\draw (-4,-1.6) node {$u^i$};

\draw (-3.7,-1.15)--(-3.4,-2);
\draw (-3.58,-1.1)--(-3.28,-1.97);
\draw (-3.3,-1.2)--(-3.05,-1.86);
\end{tikzpicture}
\caption{Illustration of the direct scattering problem in $\R^2$.}
\label{domain}
\end{figure}

We let $u=u^s+u^i$ denote the total field and $u^s$ is the scattered field created by the incident plane wave $u^i:= \text{e}^{\text{i}kx\cdot \hat{y}}$ with wave number $k>0$ and $\hat{y}$ the incident direction. The direct scattering problem for an isotropic homogeneous scatterer with a two parameter conductive boundary condition can be formulated as:
find $u^s \in H^1_{loc}(\mathbb{R}^d \backslash \partial {D} )$ satisfying
\begin{align}
\Delta u^s+k^2 n(x)u^s=k^2 \big(1-n(x)\big) u^i \quad & \text{in} \hspace{.2cm}  \mathbb{R}^d \backslash \partial D \label{direct1}\\
u^s_{-} - u_{+}^s=0  \quad \text{ and } \quad \lambda \partial_\nu \big(u^s_{-} +u^i\big) = \eta(x) \big(u_{+}^s+u^i\big) +  \partial_\nu \big(u_{+}^s+u^i\big)  \quad & \text{on} \hspace{.2cm} \partial D \label{direct2}
\end{align}
where $\partial_\nu \phi:= \nu \cdot \nabla \phi$ for any $\phi$.  Here $-$ and $+$ corresponds to taking the trace from the interior or exterior of $D$, {respectively} (see Figure \ref{domain}). To close the system, we impose the Sommerfeld radiation condition on the scattered field $u^s$ 
$$ {\partial_r u^s} - \text{i} ku^s =\mathcal{O} \left( \frac{1}{ r^{(d+1)/2} }\right) \quad \text{ as } \quad r \rightarrow \infty$$
which holds uniformly with respect to the angular variable $\hat{x}=x/r$ where $r=|x|$. Here, $| \cdot |$ denotes the Euclidean norm for a vector in $\R^d$. 

It has be shown that \eqref{direct1}--\eqref{direct2} is well-posed in \cite{fmconductbc}.  Therefore, we have that the scattered field $u^s$  has the asymptotic behavior (see for e.g. \cite{approach,IST of TA})
$$u^s(x,\hat{y})=\gamma\frac{\text{e}^{\text{i}k|x|}}{|x|^{(d-1)/2}}\left\{u^{\infty}(\hat{x},\hat{y})+\mathcal{O}\left(\frac{1}{|x|}\right)\right\}\hspace{.3cm}\text{as}\hspace{.3cm}|x|\longrightarrow \infty$$
and where the constant $\gamma$ is given by 
$$ \gamma=\frac{\text{e}^{\text{i}\pi/4}}{\sqrt{8\pi k}}\hspace{.3cm}\text{in}\hspace{.3cm}\mathbb{R}^2\hspace{.3cm}\text{and}\hspace{.3cm}\gamma=\frac{1}{4\pi}\hspace{.3cm}\text{in}\hspace{.3cm}\mathbb{R}^3.$$
Here $u^\infty(\hat{x},\hat{y})$ denotes the far-field pattern  depending on the incident direction $\hat{y}$ and the observation direction $\hat{x}$. The far-field pattern for all incident directions defines the far-field operator $F:L^2(\mathbb{S}^{d-1})\longrightarrow L^2(\mathbb{S}^{d-1})$ given by 
$$(Fg)(\hat{x}):=\int_{\mathbb{S}^{d-1}}u^{\infty}(\hat{x},\hat{y})g(\hat{y})\, \text{d}s(\hat{y}) \quad \text{ for } \quad g\in L^2(\mathbb{S}^{d-1}).$$
Here, $\mathbb{S}^{d-1}$ denotes the unit disk/sphere in $\R^d$. 
It is also well-known (see \cite{fmconductbc}) that $F$ is injective with a dense range if and only if there does not exist a {non}trivial solution $(w,v)\in H^1(D)\times H^1(D)$ solving: 
\begin{align}
\Delta w+k^2 n(x) w = 0 \quad \text{ and } \quad  \Delta v+k^2v = 0  \hspace{.2cm}& \text{in} \hspace{.2cm}  D \label{eigproblem1}\\
w=v \quad \text{ and } \quad  \lambda \partial_{\nu}w=\partial_{\nu}v+\eta (x) v \label{eigproblem2} \hspace{.2cm}& \text{on} \hspace{.2cm}  \partial D
\end{align}
where $v$ takes the form of a Herglotz function 
$$v_g({x}):=\int_{\mathbb{S}^{d-1}} \text{e}^{\text{i}k{x}\cdot \hat{y} }g(\hat{y})\, \text{d}s(\hat{y}), \hspace{2cm} g\in L^2(\mathbb{S}^{d-1}).$$

\noindent Now, the values $k\in \mathbb{C}$ for which \eqref{eigproblem1}--\eqref{eigproblem2} has nontrivial solutions are called \textit{transmission eigenvalues}. Due to the fact that, the Herglotz functions are dense in the set of solutions to Helmholtz equation we will consider the transmission eigenvalue problem for any eigenfunction $v \in H^1(D)$. Thus, the goal of this paper is to study this eigenvalue problem as well as {possible applications to} the inverse spectral problem. We first show that if a set of eigenvalues exists, then {this} will be a discrete set.

\section{Discreteness of Eigenvalues}\label{discrete}
In this section, we will study the discreteness of the transmission eigenvalues. In general, sampling methods such as the factorization method \cite{fmconductbc,kirschbook} do not provide valid reconstructions of $D$ if the wave number $k$ is a transmission eigenvalue. Here, we will assume that the conductivity parameters satisfy either: $\lambda \in (1,\infty)$ and $ \eta_{{max}} <0$ or $\lambda\in (0,1)$ and $\eta_{{min}}> 0$. Note, that due to the presence of the parameter $\lambda \neq 1$ in \eqref{eigproblem1}--\eqref{eigproblem2} the discreteness for this problem must be handled differently from the case when $\lambda = 1$ which was proven in \cite{te-cbc}. Here, we will use a different variational formulation to study \eqref{eigproblem1}--\eqref{eigproblem2}. 
To this end, we formulate the transmission eigenvalue problem as the problem for the difference $u:= w-v \in H_0^1(D)$ and $v \in H^1(D)$. By subtracting the equations and boundary conditions for $v$ and $w$ we have that the boundary value problem for $v$ and $u$ is given by
\begin{align}
\lambda (\Delta u+k^2 n u) =(1-\lambda) \Delta v +k^2(1-\lambda n) v \hspace{.2cm}& \text{in} \hspace{.2cm}  D \label{eigproblem3}\\
\lambda \partial_{\nu} u=(1-\lambda)\partial_{\nu}v+\eta v\hspace{.2cm}& \text{on} \hspace{.2cm}  \partial D. \label{eigproblem4}
\end{align}

Now, in order to analyze \eqref{eigproblem3}--\eqref{eigproblem4} we will employ a variational technique. To do so, we use Green's 1st Theorem to obtain that 
\begin{align}
\lambda \int_D \nabla u \cdot \nabla \overline{\phi}-k^2nu\overline{\phi} \, \text{d}x=\int_D (1-\lambda) \nabla v\cdot \nabla \overline{\phi}-k^2(1-\lambda n)v\overline{\phi}\text{d}x  +\int_{\partial D}\eta v\overline{\phi} \, \text{d}s \label{UtoV}
\end{align} 
for all $\phi\in H^1(D).$
In addition, we also need to enforce that $v$ is a solution to Helmholtz's equation in $D$. Therefore, by again appealing to Green's 1st Theorem, we can have that 
\begin{equation}
\int_D \nabla v \cdot \nabla\overline{\psi}\, \text{d}x=\int_D k^2 v\overline{\psi}\, \text{d}x \quad \text{for all} \,\, \psi\in H^1_0(D). \label{helmholtz}
\end{equation}
We now define the following sesquilinear forms $b( \cdot \, ,\cdot ): H^1(D)\times H^1_0(D) \longrightarrow  \C$
$$b(v,\psi)=\int_D\nabla v\cdot \nabla\overline{\psi}\, \text{d}x$$ 
and $a( \cdot \, ,\cdot ): H^1(D)\times H^1(D) \longrightarrow \C$
$$a(v,\phi)=-\frac{1}{\lambda}\int_D (1-\lambda) \nabla v\cdot \nabla \bar{\phi}\, \text{d}x-\frac{1}{\lambda}\int_{\partial D}\eta v \bar{\phi} \, \text{d}s.$$
It is clear that by appealing to the Cauchy-Schwarz inequality and the Trace Theorem that both $a( \cdot \, ,\cdot )$ and $b( \cdot \, ,\cdot )$ are bounded. Defining these sesquilinear forms helps us to {write} \eqref{eigproblem3}--\eqref{eigproblem4} as linear eigenvalue problem for the system
\begin{align}
a(v,\phi)+\overline{b(\phi,u)}&=\int_D k^2nu\overline{\phi}\, \text{d}x-\frac{1}{\lambda}\int_Dk^2(1-\lambda n)v\overline{\phi}\, \text{d}x \label{eigsystem1}\\
b(v,\psi) \hspace{0.62in} &=\int_Dk^2v\overline{\psi}\, \text{d}x. \label{eigsystem2}
\end{align}
In the analysis of the equivalent eigenvalue problem \eqref{eigsystem1}--\eqref{eigsystem2}, we will consider the corresponding source problem. Therefore, we will make the substitution $k^2v=g$ and $k^2u=f$ to define the saddle point problem corresponding to \eqref{eigsystem1}--\eqref{eigsystem2} as  
\begin{align}
a(v,\phi)+\overline{b(\phi,u)}&=\big(f,n\phi\big)_{L^2(D)}+\frac{1}{\lambda} \big(g, (\lambda n-1) \phi \big)_{L^2(D)} \label{sourceproblem1}\\
b(v,\psi) \hspace{0.62in} &=(g,\psi)_{L^2(D)}. \label{sourceproblem2}
\end{align}
It is clear that there exists constants $C_j>0$ for $j=1,2$ such that  
$$\left| \big(f,n\phi\big)_{L^2(D)}+\frac{1}{\lambda} \big(g, (\lambda n-1) \phi \big)_{L^2(D)} \right| \leq C_1 \big\{ \| f \|_{L^2(D)} +  \| g \|_{L^2(D)} \big\} \| \phi \|_{H^1(D)}$$ 
and 
$$ \big| (g,\psi)_{L^2(D)} \big| \leq C_2 \| g \|_{L^2(D)} \| \psi \|_{H^1(D)}$$ 
for all $f \in H^1_0(D)$ and $g \in H^1(D)$ since we have assumed that $n \in L^\infty (D)$.

Now, we will consider the source problem {stated} above as: given $(f,g) \in H_0^1(D) \times H^1(D)$ find $(u,v)\in H_0^1(D) \times H^1(D)$ solving \eqref{sourceproblem1}--\eqref{sourceproblem2}. Notice, that in order to prove wellposedness it is sufficient to prove that the sesquilinear form $a(\cdot \, , \cdot)$ is coercive on $H^1(D)$ and that $b(\cdot \, ,\cdot)$ has the inf--sup condition. Recall, that the inf--sup condition is defined as (see for e.g. \cite{inf-sup})
$$\inf\limits_{\psi \in H^1_0(D)}\sup\limits_{v\in H^1(D)}\frac{b(v,\psi)}{ \| \psi \|_{H^1(D)}\| v \|_{H^1(D)}}\geq \alpha$$
for some constant $\alpha>0$. In the following result, we prove that the sesquilinear forms defined above satisfy the aforementioned properties.

\begin{theorem}\label{seqformprop1}
Assuming that either $\lambda \in (1,\infty)$ and $ \eta_{{max}} <0$ or  $\lambda\in (0,1)$ and $\eta_{{min}}> 0$. Then we have that $a(\cdot, \cdot)$ is coercive on $H^1(D)$. Moreover, we have that $b(\cdot\, ,\cdot)$ satisfies the inf--sup condition.
\end{theorem}
\begin{proof}
We first show that  $a(\cdot \, , \cdot)$ is coercive and we choose to present the case where we assume that $\lambda\in (1,\infty)$ and $\eta_{\text{max}} <0$. From this, we can now estimate 
 \begin{align*}
    \lambda a(v,v)&=-\int_D (1-\lambda) |\nabla v|^2 \, \text{d}x-\int_{\partial D}\eta |v|^2 \, \text{d}s\\
    	&\geq (\lambda-1)\int_D |\nabla v|^2 \, \text{d}x-{\eta_{{max}}}\int_{\partial D} |v|^2 \, \text{d}s\\
    	&\geq \min \left\{{(\lambda-1)}, \left| {\eta_{{max}}} \right|  \right\} \left(\int_D |\nabla v|^2 \, \text{d}x+\int_{\partial D} |v|^2 \, \text{d}s \right).
\end{align*}
Now, we can use the fact that 
$$\| \cdot \|_{H^1(D)}^2 \quad \text{is equivalent to } \quad \int_D |\nabla\cdot|^2 \, \text{d}x+\int_{\partial D} |\cdot|^2 \, \text{d} s,$$
(see for e.g. \cite{salsa} Chapter 8) to obtain the estimate  
 $$| a(v,v) | \geq C \| v \|_{H^1(D)}^2\hspace{.5cm}\text{for some}\hspace{.5cm} C>0.$$ 
This proves the coercivity for the case when $\lambda \in (1,\infty)$ and $ \eta_{{max}} <0$. The case when $\lambda\in (0,1)$ and $\eta_{{min}}> 0$ can be handled in a similar manner.

In order to show that the sesquilinear form $b(\cdot \, ,\cdot)$ satisfies the inf--sup condition, we will use an equivalent definition. Recall, that the  inf--sup condition is equivalent to showing that for any $\psi \in H^1_0(D)$ there exists $v_{\psi} \in H^1(D)$ such that  
$$b(v_{\psi},\psi) \geq \beta \| \psi \|_{H^1(D)}^2$$ 
where  $\| v_{\psi} \|_{H^1(D)} \leq C\| \psi \|_{H^1(D)}$ for some constant $\beta>0$ that is independent of $\psi$. To this end, we define $v_{\psi} \in H^1(D)$ to be the solution of the variational problem 
\begin{align}
\int_D \nabla v_{\psi}\cdot \nabla\overline{\phi} \, \text{d}x+\int_{\partial D}v_{\psi}\overline{\phi} \, \text{d}s=\int_D \nabla\psi\cdot\nabla\overline{\phi} \, \text{d}x \label{infsup}
\end{align} 
for all $\phi\in H^1(D)$. By appealing to the norm equivalence stated above and the Lax-Milgram, we have that the mapping $\psi \longmapsto v_\psi$ solving \eqref{infsup}
is a well defined and bounded linear operator from $H^1_0(D)$ to $H^1(D)$ . Therefore, we have that letting $\phi=\psi$ in \eqref{infsup} gives  
$$b(v_{\psi},\psi)=\int_D \nabla v_{\psi}\cdot \nabla\overline{\psi}  \, \text{d}x=\int_D | \nabla\psi |^2  \, \text{d}x \geq \beta  \| \psi \|^2_{H^1(D)}$$
by the Poincar\'{e} inequality. Note, that we have used the fact that $\psi$ has zero trace on the boundary $\partial D$. Thus, we have that $b(\cdot \, ,\cdot)$ satisfies the inf--sup condition. 
\end{proof}

From Theorem \ref{seqformprop1} and the analysis in \cite{inf-sup} we have that \eqref{sourceproblem1}--\eqref{sourceproblem2} is wellposed. Therefore, we can define the bounded linear operator 
$$ T :H_0^1(D)\times H^1(D)\longrightarrow H_0^1(D)\times H^1(D) \quad \text{such that } \quad T(f,g) = (u,v).$$
By the wellposedness and the estimates on the $L^2(D)$ integrals on the right hand side of \eqref{sourceproblem1}--\eqref{sourceproblem2} we have that for some $C>0$ 
$$\| T(f,g) \|_{H^1(D) \times H^1(D)} =\| (u,v) \|_{H^1(D) \times H^1(D)} \leq C  \big\{ \| f \|_{L^2(D)} +  \| g \|_{L^2(D)} \big\}.$$
Now, we have {the necessary requirements} to prove that the solution operator $T$ is compact using the Rellich--Kondrachov Embedding Theorem. 

\begin{theorem}\label{Tcompact}
Assuming that either $\lambda \in (1,\infty)$ and $ \eta_{\text{max}} <0$ or $\lambda\in (0,1)$ and $\eta_{\text{min}}> 0$.
Then the solution operator $T:H_0^1(D)\times H^1(D)\longrightarrow H_0^1(D)\times H^1(D)$ corresponding to  \eqref{sourceproblem1}--\eqref{sourceproblem2} is compact.
\end{theorem}
\begin{proof}
To prove the claim, we will show that for any sequence $(f_j ,  g_j)$ weakly converging to zero in $H_0^1(D)\times H^1(D)$ then the image $T(f_j , g_j)$ has a subsequence that converges strongly to zero in $H_0^1(D)\times H^1(D)$. Notice, that there exists a subsequence (still denoted with $j$) that satisfies  
$$\| f_j \|_{L^2(D)} +  \| g_j \|_{L^2(D)} \to 0 \quad \text{as} \quad j \to \infty$$
by the compact embedding of $H^1(D)$ in $L^2(D)$ see \cite{evans}. 
From this, we have that  
$$\| T(f_j , g_j) \|_{H^1(D) \times H^1(D)} \leq C  \big\{ \| f_j \|_{L^2(D)} +  \| g_j \|_{L^2(D)} \big\}\to 0 \quad \text{as} \quad j \to \infty$$
which proves the claim. 
\end{proof}


Now, simple calculations show that the relationship between the eigenvalues of $T$ and the transmission eigenvalues $k$ is that $1/k^2 \in \sigma(T)$, where $\sigma(T)$ is the spectrum of the operator $T$. Therefore, we have related the transmission eigenvalues to the eigenvalues of a compact operator. We can use the compactness of $T$ to prove the following result for the set of transmission eigenvalues independent of the sign of the contrast $n-1$. 
\begin{theorem}\label{TEdiscrete}
Assuming that either $\lambda \in (1,\infty)$ and $ \eta_{\text{max}} <0$ or  $\lambda\in (0,1)$ and $\eta_{\text{min}}> 0$.
Then the set of transmission eigenvalues is discrete with no finite accumulation point.
\end{theorem}
\begin{proof}
This is a consequence of the fact that $k$ is a transmission eigenvalue implies that $1/k^2 \in \sigma(T)$. Then we exploit that the set $\sigma(T)$ is a discrete set with zero its only possible accumulation point.
\end{proof}

An important question is whether or not the operator $T$ is self-adjoint. If so, we would have existence of real transmission eigenvalues by appealing to the Hilbert-Schmidt Theorem. {In a similar way with other} transmission eigenvalue problems we have that the operator $T$ is not self-adjoint even when the material parameters are real-valued. To see this fact, we can consider the transmission eigenvalue problem for the unit disk in $\R^2$ with constant coefficients $\lambda$, $\eta$ and $n$.
\begin{example}\label{ex1}
Using separation of variables we have that $k$ is a transmission eigenvalue provided that $d_m(k) = 0$ for any $m \in \Z$ where 
$$d_m(k):=\text{det} \begin{pmatrix}
J_m(k\sqrt{n}) &-J_m(k)\\
\lambda J^{'}_{m}(k\sqrt{n})k\sqrt{n} & - \big(kJ^{'}_{m}(k)+\eta J_m(k)\big)
\end{pmatrix}$$
and $J_m(t)$ are the Bessel functions of the first kind of order $m$ (see Section \ref{numerics} for details). Therefore, we can plot $|d_0(k)|$ for complex-valued $k$ and determine if there are any complex roots. This is done in Figure \ref{TEcomplex} using  $\lambda=2$, $n=4$, and $\eta=-\frac{1}{100}$. We see complex roots at the values $k=2.2032 \pm 0.2905\mathrm{i}$ as well as other points in the set $[0, 10]\times[-1, 1]\mathrm{i}$.

\begin{figure}[ht]
\centering 
\includegraphics[width=10cm]{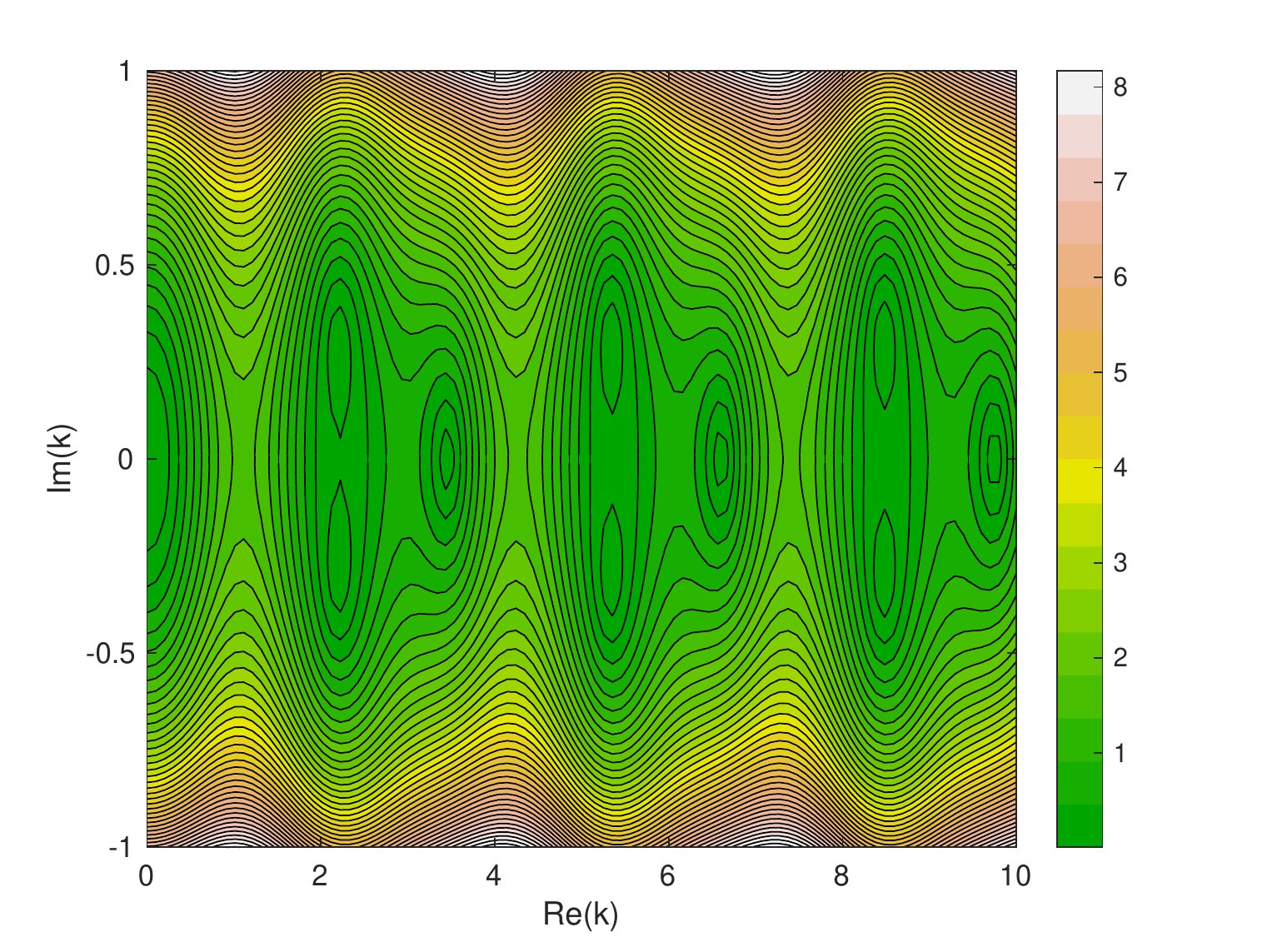}
\caption{Contour plot of $|d_0(k)|$ on the set $[0, 10]\times[-1, 1]\text{i}$ in the complex plane where the parameters are  $\lambda=2$, $n=4$ and $\eta=-\frac{1}{100}$.}
\label{TEcomplex}
\end{figure}

More precisely, we obtain ten interior transmission eigenvalues within the given set for $m=0$ with \texttt{MATLAB}. They are given to high accuracy as $0.053410$, $2.203160\pm 0.290468\mathrm{i}$, $3.456704$, $5.338551 \pm 0.305549\mathrm{i}$, $6.606526$, $8.477827\pm 0.309699\mathrm{i}$, and $9.750981$. 
\end{example}

%

From this, we see that there are multiple complex transmission eigenvalues $k$ for this set of parameters. {As a result,} for this simple example, $T$ has complex eigenvalues since $1/k^2 \in \sigma(T)$ and cannot be self-adjoint. Therefore, we can not rely on standard theory to prove the existence of the transmission eigenvalues. The existence is proven in the next section where we use similar analysis as in \cite{C and K}. These techniques are usually used for anisotropic materials. This analysis is {utilized} due to the fact that the techniques in \cite{te-cbc} fail to give a variational formulation for the eigenfunction $u=w-v$ exclusively.

\section{Existence of Transmission Eigenvalues}\label{exist}
\noindent In this section, we will show the existence of the transmission eigenvalues with conductive boundary parameters following a similar analysis as \cite{C and K}. In our analysis, we will {furthermore} assume that $\lambda \in (1,\infty)$, $ \eta_{max} <0$, and $\lambda n_{max}-1<0$, or $\lambda\in (0,1)$, $\eta_{min}> 0$, and $\lambda n_{min}-1>0$. The goal now is to show the existence of real transmission eigenvalues. To this end, we work with the formulated problem \eqref{eigproblem3}--\eqref{eigproblem4} and the variational formulation \eqref{UtoV} 
\begin{align*}
\lambda\int_D \nabla u \cdot \nabla \overline{\phi}-k^2nu\overline{\phi} \, \text{d}x=\int_D (1-\lambda) \nabla v\cdot \nabla \overline{\phi}-k^2(1-\lambda n)v\overline{\phi} \, \text{d}x+\int_{\partial D}\eta v\overline{\phi}  \, \text{d}s
\end{align*}
for all $\phi\in H^1(D).$ Following the analysis in \cite{C and K}, we will consider \eqref{eigproblem3}--\eqref{eigproblem4} as a Robin boundary value problem for $v \in H^1(D)$. This means that for a given  $u\in H^1_0(D)$ we need to show that there exists a  $v \in H^1(D)$ satisfying \eqref{UtoV} . We now define the bounded sesquilinear form and the bounded conjugate linear functional from the variational formulation as
$$A(v, \phi)=\int_{\partial D}\eta v \bar{\phi}  \, \text{d}s + \int_D (1-\lambda) \nabla v\cdot \nabla \bar{\phi}-k^2(1-\lambda n)v\bar{\phi}  \, \text{d}x$$
and 
$$\ell(\phi) =\lambda\int_D \nabla u \cdot \nabla \bar{\phi} -k^2nu\bar{\phi}  \, \text{d}x.$$
Applying the Lax-Milgram Lemma to $A(v, \phi)=\ell(\phi)$ gives us that \eqref{eigproblem3}--\eqref{eigproblem4}  is wellposed i.e. there exists a unique solution $v\in H^1(D)$  satisfying \eqref{eigproblem3}--\eqref{eigproblem4} for any given $u\in H^1_0(D)$. Notice, that the coercivity result for  $A(v, \phi)$ is proven in a similar manner as $a(\cdot,\cdot)$ in Section \ref{discrete}. This says that the mapping we have $u \longmapsto v_u$ from $H_0^1(D)$ to $H^1(D)$ is a bounded linear operator. Because the transmission eigenfunction $v$ solves the Helmholtz equation in $D$, we make sure that $v_u$ is also a solution of the Helmholtz equation in the variational sense. To this end, we use the Riesz Representation Theorem to define $\mathbb{L}_k u$ by 
\begin{equation}
\big(\mathbb{L}_k u, \psi \big)_{H^1(D)}=\int_D\nabla v_u\cdot \nabla\overline{\psi}-k^2 v_u\overline{\psi} \, \text{d}x  \hspace{.4cm}\forall \psi\in H^1_0(D).
\end{equation}
Notice, that $\mathbb{L}_k u=0$ if and only if $v_u$ solves the Helmholtz equation.

We will analyze the null-space of the operator $\mathbb{L}_k:H_0^1(D)\longrightarrow H_0^1(D)$ and connect this to the set of transmission eigenfunctions. To this end, we show that $\mathbb{L}_k$ having a non-trival null-space for a given value of $k$ is equivalent to the transmission eigenvalue problem  \eqref{eigproblem1}--\eqref{eigproblem2}.

\begin{theorem}
 Assume that either  $\lambda \in (1,\infty)$, $ \eta_{max} <0$, and $\lambda n_{max}-1<0$ or $\lambda\in (0,1)$, $\eta_{min}> 0$, and $\lambda n_{min}-1>0$. If $(v,w) \in H^1(D) \times H^1(D)$ are non-trivial solutions of \eqref{eigproblem1}--\eqref{eigproblem2} then the non-trivial  $u=w-v \in H^1_0 (D)$ satisfies that $\mathbb{L}_k u=0$. Conversely, if for a given value of $k$ we have that $\mathbb{L}_k u=0$ for a non-trivial $u \in H^1_0(D)$, then $v_u \in H^1(D)$ solving \eqref{eigproblem3}--\eqref{eigproblem4} and $w=u+v_u$ are non-trivial solutions of \eqref{eigproblem1}--\eqref{eigproblem2}.
\end{theorem}
\begin{proof}
The first part of the theorem is given by our construction. Conversely, we assume $\mathbb{L}_k u=0$ for a given value of $k$ provided that $u \neq 0$ and we let $v=v_u\in H^1(D)$ be the unique solution to  \eqref{eigproblem3}--\eqref{eigproblem4}, then define $w=u+v \in H^1(D)$. From equation \eqref{eigproblem3} along with the fact that $\mathbb{L}_k u=0$ gives that  
$$\Delta v +k^2 v=0 \quad \text{ and} \quad \Delta w+k^2n w=0 \quad \text{ in } \,\, D.$$ 
Similarly, {from} the boundary condition \eqref{eigproblem4} given by  $\lambda\partial_{\nu}u = (1-\lambda)\partial_{\nu}v+\eta v$ on $\partial D$ and using the identity $w=u+v$ we can easily obtain that 
$$ \lambda \partial_\nu w =  \partial_\nu v + \eta v \quad \text{on } \,\, \partial D.$$
This proves the claim since $ u \in H^1_0(D)$. 
\end{proof}

We have shown that there exists transmission eigenvalues if and only if the null-space of $\mathbb{L}_k$ is non-trivial. Therefore, we turn our attention to studying this operator. Now, we are going to highlight some properties of the operator $\mathbb{L}_k$ that will help us establish when $\mathbb{L}_k$ has a trivial null-space. From here on, we denote $v_u:=v$.

\begin{theorem}\label{Thm4.2}
Assume that either  $\lambda \in (1,\infty)$, $ \eta_{max} <0$, and $\lambda n_{max}-1<0$ or $\lambda\in (0,1)$, $\eta_{min}> 0$, and $\lambda n_{min}-1>0$. Then, we have the following:
\begin{enumerate}
\item the operator $\mathbb{L}_k:H^1_0(D)\longrightarrow H^1_0(D)$ is self-adjoint,
\item the operator $-\mathbb{L}_0$ or $\mathbb{L}_0$ is coercive when $\lambda\in(1,\infty)$ or $\lambda\in (0,1)$, respectively. 
\item and the operator  $\mathbb{L}_k-\mathbb{L}_0$ is compact. 
\end{enumerate}
\end{theorem}
\begin{proof}
(1) Now we show that the operator $\mathbb{L}_k$ is self-adjoint. To this end, it is enough to show that the quantity 
$$(\mathbb{L}_ku, u)_{H^1(D)}=\int_D\nabla v\cdot \nabla\overline{u}-k^2 v\overline{u} \, \text{d}x $$
is real-valued for all $u$ (see for e.g. \cite{axler}). 
Recall, the variational formulation given by \eqref{UtoV} 
$$\lambda\int_D \nabla u \cdot \nabla \bar{\phi} -k^2nu\bar{\phi}  \, \text{d}x=\int_{\partial D}\eta v \bar{\phi}  \, \text{d}s +\int_D (1-\lambda) \nabla v\cdot \nabla \bar{\phi}-k^2(1-\lambda n)v\bar{\phi}  \, \text{d}x$$
for any $\phi \in H^1(D)$. Letting $\phi=u$ in \eqref{UtoV} implies that
\begin{equation}
 \lambda\int_D |\nabla u|^2-k^2n|u|^2  \, \text{d}x= \int_D (1-\lambda) \nabla v\cdot \nabla \bar{u}-k^2(1-\lambda n)v\overline{u} \label{phi=u}  \, \text{d}x.
 \end{equation}
In a similar manner, letting $\phi=v$ in the variational formulation \eqref{UtoV} we obtain 
\begin{equation}
\lambda\int_D \nabla u \cdot \nabla \overline{v}-k^2nu\overline{v}  \, \text{d}x=\int_{\partial D}\eta |v|^2  \, \text{d}s +\int_D (1-\lambda) |\nabla v|^2-k^2(1-\lambda n)|v|^2  \, \text{d}x. \label{phi=v}\end{equation}
By the definition of $\mathbb{L}_k$ we have that
\begin{align*}
    (\mathbb{L}_ku, u)_{H^1(D)}&=\int_D\nabla v\cdot \nabla\overline{u}-k^2 v\overline{u} \, \text{d}x\\
    &=\int_D(1-\lambda)\nabla v \cdot \nabla \overline{u}-k^2(1-\lambda n)v \overline{u} \, \text{d}x +\lambda \int_D\nabla v\cdot \nabla\overline{u}- k^2 n v\overline{u} \, \text{d}x.
\end{align*}
Using \eqref{phi=u} and \eqref{phi=v} above we obtain that 
\begin{align*}
    (\mathbb{L}_ku, u)_{H^1(D)}=\lambda\int_D |\nabla u|^2&-k^2n|u|^2  \, \text{d}x+\int_{\partial D}\eta |v|^2  \, \text{d}s \\
   & +\int_D (1-\lambda) |\nabla v|^2-k^2(1-\lambda n)|v|^2  \, \text{d}x.
\end{align*}
Thus all the integrals on the right hand side are evaluated to be real numbers and that gives us that $\mathbb{L}_k$ is self-adjoint.\\
\newline (2) Now, we show that $\pm\mathbb{L}_0$ is coercive and we first analyze $-\mathbb{L}_0$. We assume that $\lambda\in (1,\infty)$ and that $\eta_{max}<0$. Letting $w=v+u$ in the definition of $\mathbb{L}_k$ gives

\begin{align*}
    (\mathbb{L}_ku,u)_{H^1(D)}=\int_D\nabla w\cdot\nabla \overline{u} -k^2 w \overline{u} \, \text{d}x-\int_D|\nabla u|^2 -k^2 |u|^2 \, \text{d}x.
\end{align*}
From the variational formulation \eqref{UtoV} with $\phi=w$ we have the following equality 
\begin{equation}
    \int_D\nabla w\cdot \nabla\overline{u}-k^2 w\overline{u} \, \text{d}x=\int_D (1-\lambda)|\nabla w|^2-k^2(1-\lambda n)|w|^2  \, \text{d}x+\int_{\partial D}\eta |w|^2  \, \text{d}s. \label{phi=w}\end{equation}
Now, using \eqref{phi=w} we get 
\begin{align}
    (\mathbb{L}_ku,u)_{H^1(D)}=\int_D(1-\lambda)|\nabla w|^2&-k^2(1-\lambda n)|w|^2 \, \text{d}x \nonumber \\
    &+\int_{\partial D}\eta |w|^2 \, \text{d}s-\int_D|\nabla u|^2 -k^2 |u|^2 \, \text{d}x.
    \end{align}
Therefore, letting $k=0$ we obtain 
\begin{equation}
    -(\mathbb{L}_0u,u)_{H^1(D)}=\int_D(\lambda-1)|\nabla w|^2  \, \text{d}x-\int_{\partial D}\eta |w|^2 \, \text{d}s+\int_D|\nabla u|^2  \, \text{d}x .\end{equation}
By appealing to the assumptions $\lambda\in (1,\infty)$ and $\eta_{max}<0$ we see that 
$$\int_D(\lambda-1)|\nabla w|^2 \, \text{d}x\geq 0\hspace{.5cm}\text{and}\hspace{.5cm} - \int_{\partial D}\eta |w|^2 \, \text{d}s\geq 0.$$
From this we can estimate  
\begin{align*}
    -(\mathbb{L}_0u,u)_{H^1(D)}&=\int_D(\lambda-1)|\nabla w|^2 \, \text{d}x-\int_{\partial D}\eta |w|^2\text{d}s+\int_D|\nabla u|^2  \, \text{d}x\\
    &\geq \int_D|\nabla u|^2  \, \text{d}x=\norm{\grad u}^2_{L^2 (D) }
\end{align*} 
proving the coercivity of the $-\mathbb{L}_k$ operator in $H^1_0(D)$.\\ 

 Next, assume that $\lambda\in (0,1)$ and $ \eta_{min}>0$ and for this case we consider the operator $\mathbb{L}_0$. From the definition of $\mathbb{L}_k$ we have that
$$(\mathbb{L}_k u,u)_{H^1(D)}=\int_D \nabla v \cdot \nabla  \overline{u}-k^2 v \overline{u} \, \text{d}x.$$
\noindent Letting $k=0$ in the variational formulation \eqref{UtoV} with $\phi =u$ gives 
\begin{equation} 
\lambda \int_D |\nabla u|^2 \, \text{d}x=\int_D(1-\lambda)\nabla v\cdot \nabla \overline{u} \, \text{d}x. \label{k=0}
\end{equation} 
In a similar way, using that $k=0$ in the variational formulation \eqref{UtoV} with $\phi =v$ gives us
\begin{equation}
\lambda \int_D \nabla u \cdot \nabla \overline{v} \, \text{d}x=\int_D(1-\lambda)|\nabla v|^2  \, \text{d}x+\int_{\partial D} \eta |v|^2  \, \text{d}s. \label{v,k=0} \end{equation}
Now, consider $\mathbb{L}_0$ and using  \eqref{k=0} and \eqref{v,k=0} provide independently, and so we get
\begin{align*}
    (\mathbb{L}_0u,u)_{H^1(D)}&=\int_D \nabla v\cdot \nabla \overline{ u} \, \text{d}x\\
&=\lambda\int_D |\nabla u|^2 \text{d}x +\int_{\partial D}\eta |v|^2 \text{d}s +\int_D(1-\lambda)|\nabla v|^2 \, \text{d}x\\
&\geq\lambda\int_D |\nabla u|^2 \text{d}x +\int_{\partial D}\eta_{min} |v|^2 \text{d}s +\int_D(1-\lambda)|\nabla v|^2 \, \text{d}x\\
&\geq \lambda \norm{\nabla u}^2_{L^2(D)}
\end{align*}
\sout{and} where we have used the assumptions of $\lambda\in (0,1)$ and $\eta_{min}>0$.  Therefore, proving the coercivity in this case.  \\
\newline 
 $(3)$ Now, we turn our attention to proving the compactness of $\mathbb{L}_k-\mathbb{L}_0$. To do so, we assume that we have a {weakly convergent} sequence $u^j\rightharpoonup 0$ in $ H_0^1(D)$. By the wellposedness, there exists $v^j_k\rightharpoonup 0$ and $v_0^j\rightharpoonup 0$ in $H^1(D),$ where these correspond to the solutions of our variational formulation \eqref{UtoV}. The definition of  $\mathbb{L}_k$ gives us that we can define $(\mathbb{L}_k-\mathbb{L}_0)u^j$ in terms of $v_k^j$ and $v_0^j$. Using the variational formulation \eqref{UtoV}, we have that   
 $$\int_{\partial D}\eta v_k^j\overline{\phi} \, \text{d}s+\int_D(1-\lambda)\nabla v_k^j\cdot \nabla \overline{\phi}-k^2(1-\lambda n)v_k^j\overline{\phi} \, \text{d}x=\lambda\int_D\nabla u^j\cdot \nabla \overline{\phi}-k^2 n u^j\overline{\phi} \, \text{d}x$$
 and 
 $$\int_{\partial D}\eta v_0^j\overline{\phi} \, \text{d}s+\int_D(1-\lambda)\nabla v_0^j\cdot \nabla \overline{\phi} \, \text{d}x=\lambda\int_D\nabla u^j\cdot \nabla \overline{\phi} \, \text{d}x$$
for all $\phi \in H^1(D)$. Subtracting both equations gives us that
$$\int_{\partial D}\eta (v_k^j-v_0^j)\overline{\phi} \, \text{d}s+\int_D(1-\lambda)\nabla  (v_k^j-v_0^j)\cdot \nabla \overline{\phi} \, \text{d}x=\int_Dk^2(1-\lambda n)v_k^j\overline{\phi}-\lambda n k^2u^j\overline{\phi} \, \text{d}x.$$
We now let $\phi= v_k^j-v_0^j$ and  we have the following
$$\int_{\partial D}\eta | v_k^j-v_0^j|^2 \, \text{d}s+\int_D(1-\lambda)|\nabla (v_k^j-v_0^j) |^2 \, \text{d}x=\int_Dk^2(1-\lambda n)v_k^j\overline{(v_k^j-v_0^j)}-\lambda nk^2u^j\overline{(v_k^j-v_0^j)} \, \text{d}x.$$
Notice, that on the left hand side we use the fact that 
$$\| \cdot \|_{H^1(D)}^2 \quad \text{is equivalent to } \quad \int_D |\nabla\cdot|^2 \, \text{d}x+\int_{\partial D} |\cdot|^2 \, \text{d} s.$$
By the compact embedding of $H^1(D)$ into $L^2(D)$, we have that $v_k^j$ and $u^j$ converge strongly to zero in the $L^2(D)$--norm. 
Thus, we have that the right hand side behaves as
$$ \|{v_k^j-v_0^j}\|_{H^1(D)}\leq  C\Big(\|{v^j_k}\|_{L^2(D)}+\|{u^j}\|_{L^2(D)} \Big) \longrightarrow 0$$
as $j \to \infty$. Notice that the $C>0$ above is independent of the parameter $j$'s but does depend on the material parameters. Note that we have used the assumptions on $\lambda$ and $\eta$. Now, observe the following 
\begin{align*}
    \Big((\mathbb{L}_k-\mathbb{L}_0)u^j,\psi\Big)_{H^1(D)}&=\int_D \nabla v_k^j\cdot \nabla \overline{\psi}-k^2v_k^j\overline{\psi}  \, \text{d}x- \int_D\nabla v_0^j\cdot \nabla \overline{\psi} \, \text{d}x\\
    &=\int_D\nabla(v_k^j-v_0^j) \cdot \nabla \overline{\psi}-k^2 v_k^j\overline{\psi} \, \text{d}x
\end{align*}
and using the Cauchy-Schwartz inequality we have
$$\|{(\mathbb{L}_k-\mathbb{L}_0)u^j}\|_{H^1(D)}\leq C\Big(\|{v^j_k}\|_{L^2(D)}+\|{v_k^j-v_0^j}\|_{H^1(D)} \Big) \longrightarrow 0 . $$ 
Therefore, we have shown that $(\mathbb{L}_k-\mathbb{L}_0)u^j$ tends to zero as $j$ tends to infinity, proving the claim.
\end{proof}

We have shown three important properties that will help us establish when our operator $\mathbb{L}_k$ has a trivial null-space. In addition, we want to make the observation that $\mathbb{L}_k$ depends continuously on $k$ by a similar argument as in Theorem \ref{Thm4.2}. We continue by showing that the operator $\pm \mathbb{L}_k$ is positive for a range of values which will give a lower bound on the transmission eigenvalues. 
\begin{theorem}\label{T4.3}
Let $\mu_1(D$) be the first Dirichlet eigenvalue of $-\Delta$ and let $k^2$ be a real transmission eigenvalue.  Then, we have the following:
\begin{enumerate}
\item If $\lambda \in (1,\infty)$, $ \eta_{max} <0$, and $\lambda n_{max}-1<0$, then $-\mathbb{L}_k$ is a positive operator for $k^2<\mu_1(D)$.
\item If $\lambda\in (0,1)$, $\eta_{min}> 0$, and $\lambda n_{min}-1>0$, then $\mathbb{L}_k$ is a positive operator for $k^2< \frac{\mu_1(D)}{n_{max}}$.
\end{enumerate}
\end{theorem}
\begin{proof} $(1)$ We first assume that $\lambda \in (1,\infty)$, $ \eta_{max} <0$, and $\lambda n_{max}-1<0$. Using the definition of $\mathbb{L}_k$ and $w=v+u$ we have that 
\begin{align*}
    -(\mathbb{L}_ku,u)_{H^1(D)}&=-\int_D(1-\lambda)|\nabla w|^2-k^2(1-\lambda n)|w^2| \, \text{d}x-\int_{\partial D}\eta|w^2| \, \text{d}s\\
      &\hspace{5cm}+\int_D |\nabla u|^2-k^2|u|^2 \, \text{d}x\\
    & \geq \int_D(\lambda-1)|\nabla w|^2-k^2(\lambda n_{max}-1)|w^2| \, \text{d}x-\eta_{max}\int_{\partial D}|w^2| \, \text{d}s\\
    &\hspace{5cm}+\int_D |\nabla u|^2-k^2|u|^2 \, \text{d}x\\
    &\geq  \int_D |\nabla u|^2-k^2|u|^2 \, \text{d}x.
\end{align*}
Observe, that $u\in H^1_0(D)$  implies that we have the estimate 
$$\norm{u}_{L^2(D)}^2\leq \frac{1}{\mu_1(D)}\norm{\nabla u}_{L^2(D)}^2 \quad \text{(i.e. Poincar\'{e} inequality)}$$
 where $\mu_1(D)$ is the 1st Dirichlet eigenvalue of $-\Delta$. This gives that
\begin{align*}
    -(\mathbb{L}_ku,u)_{H^1(D)} &\geq \left( 1-\frac{k^2}{\mu_1(D)}\right)\norm{\grad u}^2_{L^2(D)}.
\end{align*}
Now if $\left( 1-\frac{k^2}{\mu_1(D)} \right)> 0$, we have that $-(\mathbb{L}_k u,u)_{H^1(D)}>0$ for all $u \neq 0$ which gives us that all real transmission eigenvalues must satisfy that $k^2\geq \mu_1(D).$\\
\newline 
$(2)$ On the other hand, assume that  $\lambda\in (0,1)$, $\eta_{min}> 0$, and $\lambda n_{min}-1>0$. Using our variational formulation \eqref{UtoV} and let $\phi=u$ to obtain  
\begin{align*}
   (\mathbb{L}_ku,u)_{H^1(D)}&=\lambda\int_D|\nabla u|^2-k^2n|u|^2 \, \text{d}x + \int_D(1-\lambda)|\nabla v|^2-k^2(1-\lambda n)|v|^2  \, \text{d}x \\
    &\hspace{5cm}+\int_{\partial D} \eta |v|^2  \, \text{d}s\\
   & \geq \lambda\int_D|\nabla u|^2-k^2n_{max}|u|^2 \, \text{d}x+ \int_D(1-\lambda)|\nabla v|^2+k^2(\lambda n_{min}-1)|v|^2  \, \text{d}x\\
   &\hspace{5cm}+\int_{\partial D} \eta_{min} |v|^2  \, \text{d}s
\end{align*}
\begin{equation}\label{26}
    \hspace{1.4cm}\geq  \lambda\int_D|\nabla u|^2-k^2n_{max}|u|^2 \, \text{d}x.
\end{equation}
By again, appealing to the Poincar\'{e} inequality we have that 
\begin{align*}
    (\mathbb{L}_ku,u)_{H^1(D)} \geq \lambda \left( 1- n_{max} \frac{k^2}{\mu_1(D)} \right) \norm{\grad u}^2_{L^2(D)}.
\end{align*}
Now if $\left( 1-n_{max}\frac{k^2}{\mu_1(D)} \right)> 0$, we {conclude} that $(\mathbb{L}_k u,u)_{H^1(D)}>0$ for all $u \neq 0$ which implies that all real transmission eigenvalues must satisfy that $k^2\geq \frac{\mu_1(D)}{n_{max}}.$
\end{proof}
Theorem \ref{T4.3} shows that the operator $\pm\mathbb{L}_k$ is positive for a range of $k$ values. Next, we will show one last result to help us establish when the null-space of $\mathbb{L}_k$ is nontrivial. The property that we want to show is that the operator $\pm\mathbb{L}_k$ is non-positive for some $k$ on a subset of $H_0^1(D)$.
\begin{theorem}\label{T4.4}
There exists $\tau>0$ such that $-\mathbb{L}_\tau$, or $\mathbb{L}_\tau$ for  $\lambda \in (1,\infty)$, $ \eta_{max} <0$, and $\lambda n_{max}-1<0$, or $\lambda\in (0,1)$, $\eta_{min}> 0$, and $\lambda n_{min}-1>0$ respectively, is non-positive on $N$--dimensional subspaces of $H_0^1(D)$ for any $N \in \N$.
\end{theorem}
\begin{proof} We begin with the case when $\lambda \in (1,\infty)$, $ \eta_{max} <0$, and $\lambda n_{max}-1<0$. We consider the ball $B_\epsilon$ of radius $\epsilon>0$ such that   $B_\epsilon\subset D$. Using separation of variables one can see that there exist transmission eigenvalues for the system (See Section \ref{numerics}) 
\begin{align*}
\Delta w_1+\tau^2 n_{max} w_1 = 0 \quad \text{ and } \quad \Delta v_1+\tau^2v_1 = 0 \quad  \text{in} \quad  B_\epsilon\\ 
w_1=v_1 \quad \text{ and } \quad  \lambda \partial_{\nu} w_1=\partial_{\nu}v_1  \quad \text{on} \quad  \partial B_\epsilon .
\end{align*}
Letting $u_1$ be the difference of the eigenfunctions with corresponding eigenvalue $\tau$ gives us the following using \eqref{26}
$$ \int_{B_\epsilon} |\nabla u_1|^2- \tau^2 |u_1|^2 \, \text{d}x+\int_{B_\epsilon} (\lambda-1) |\nabla w_1|^2 -  \tau^2(\lambda n_{max}-1)|w_1|^2 \, \text{d}x=0 .$$
Therefore, since  $u_1\in H^1_0(B_\epsilon)$ we can take the extension by zero of $u_1$ to the whole domain be denoted by $u_2 \in H^1_0(D)$. Now, since $\lambda \in (1,\infty)$, $ \eta_{max} <0$, and $\lambda n_{max}-1<0$, we can construct the nontrivial $v_2\in H^1(D)$ that solves the variational formulation \eqref{UtoV} with coefficients $\lambda$, $n$, and $\eta$ in the domain $D$ and we also let $w_2=v_2+u_2$. Using the relationship between $v_2$ and $u_2$ and $w_2=v_2+u_2$ just as in the proof of Theorem \ref{Thm4.2} we have that
\begin{align}
    \int_D(\lambda-1)\nabla w_2 \cdot \nabla \overline{\phi}-\tau^2(\lambda n-1)w_2\overline{\phi} \, \text{d}x-\int_{\partial D}\eta w_2\overline{\phi} \, \text{d}s &=-\int_{D}\nabla u_2 \cdot \nabla\overline{\phi}-\tau^2u_2\overline{\phi} \, \text{d}x \nonumber \\ 
    &\hspace{-1.5in} =-\int_{B_\epsilon}\nabla u_1\cdot \nabla\overline{\phi}-\tau^2u_1\overline{\phi} \, \text{d}x \nonumber \\
     &\hspace{-1.5in}=\int_{B_\epsilon}(\lambda-1)\nabla w_1 \cdot \nabla \overline{\phi}-\tau^2(\lambda n_{max}-1)w_1\overline{\phi} \, \text{d}x. \label{27}
\end{align}
Letting $\phi=w_2$ in \eqref{27} and using the Cauchy-Schwartz inequality because we have an inner product on the right hand side over the space $H^1(B_\epsilon)$  gives us
\begin{align*}
\int_D(\lambda-1)|\nabla w_2|^2-\tau^2(\lambda n-1)|w_2|^2 \, \text{d}x&-\int_{\partial D}\eta| w_2|^2 \, \text{d}s\\
&=\int_{B_\epsilon}(\lambda-1)\nabla w_1 \cdot \nabla \overline{w_2}-\tau^2(\lambda n_{max}-1)w_1\overline{w_2} \, \text{d}x
\end{align*}
$$\leq \Big[\int_{B_\epsilon}(\lambda-1)|\nabla w_1|^2 -\tau^2(\lambda n_{max}-1)|w_1|^2 \, \text{d}x\Big]^{\frac{1}{2}}\Big[\int_{B_\epsilon}(\lambda-1)|\nabla w_2|^2 -\tau^2(\lambda n_{max}-1)|w_2|^2 \, \text{d}x\Big]^{\frac{1}{2}}.$$
As a consequence of the above inequality we have that 
\begin{align*}
\int_D(\lambda-1) | \nabla w_2 |^2-\tau^2(\lambda n-1) |w_2|^2 \, \text{d}x&-\int_{\partial D}\eta | w_2 |^2  \, \text{d}s \\
&\leq \int_{B_\epsilon}(\lambda-1)|\nabla w_1|^2 -\tau^2(\lambda n_{max}-1)|w_1|^2 \, \text{d}x.
\end{align*}
Now, we use the definition of $-\mathbb{L}_\tau$ in \eqref{26} with the functions $u_2$ and $w_2$ to conclude that
\begin{align*}
    -(\mathbb{L}_{\tau}u_2,u_2)_{H^1(D)}&= - \int_D\nabla v_2\cdot \nabla \overline{u_2}-\tau^2 v_2 \overline{u_2}  \, \text{d}x\\ 
    &=\int_{D} |\nabla u_2|^2- \tau^2 |u_2|^2 \, \text{d}x \\
     &\hspace{0.5in} +\int_{D} (\lambda-1) |\nabla w_2|^2 -  \tau^2(\lambda n-1)|w_2|^2 \, \text{d}x-\int_{\partial D} \eta |w_2|^2 \, \text{d}s 
\end{align*}
by the calculations in Theorem \ref{T4.3}. Now, estimating using the above inequality to obtain      
\begin{align*}
    (\mathbb{L}_{\tau}u_2,u_2)_{H^1(D)}&=\int_{B_\epsilon} |\nabla u_1|^2- \tau^2 |u_1|^2 \, \text{d}x \\
    &\hspace{0.5in}+\int_{D} (\lambda-1) |\nabla w_2|^2 -  \tau^2(\lambda n-1)|w_2|^2 \, \text{d}x-\int_{\partial D} \eta |w_2|^2 \, \text{d}s\\
    &\leq \int_{B_\epsilon} |\nabla u_1|^2- \tau^2 |u_1|^2 \, \text{d}x +\int_{B_\epsilon}(\lambda-1)|\nabla w_1|^2 -\tau^2(\lambda n_{max}-1)|w_1|^2 \, \text{d}x\\
    &=0.
\end{align*}
Thus, the operator is non-positive on this one dimensional subspace. 

We now argue that, for some $\tau >0$,  we can construct an $N$--dimensional subspace of $H^1_0(D)$ where the operator $-\mathbb{L}_\tau$ is non-positive for any $N \in \N$. To this end, let $N$ be fixed and define $B_j=\{ B(x_j,\epsilon):x_j\in D, \epsilon>0 \} \subset D$ for $j=1, \cdots,N$ where we assume $B_j\cap B_i=\emptyset$ for all $i \neq j$. We make the assumption that $\lambda \in (1,\infty)$, $ \eta_{max} <0$, and $\lambda n_{max}-1<0$ and denoting $\tau$ as the smallest transmission eigenvalue for 
$$\Delta w_j+\tau^2 n_{max} w_j = 0 \quad \text{ and } \quad \Delta v_j+\tau^2v_j = 0 \quad  \text{in} \quad  B_j$$
$$w_j=v_j \quad \text{ and } \quad  \lambda \partial_{\nu} w_j=\partial_{\nu}v_j \quad \text{on} \quad  \partial B_j .$$ 
From this, we let  $u_j\in H^1_0(D)$ be the difference of the eigenfunctions $w_j$ and $v_j$ extended to $D$ by zero. Therefore, we have that for $j=1,\dots,N$ the supports of $u_j$ and $u_i$ are disjoint, i.e. $u_j$ and $u_i$ are orthogonal to each other for $j\neq i$. Thus, the span$\{u_1,u_2,\dots,u_{N}\}$ is a $N$--dimensional subspace of $H^1_0(D)$. Now, because the support of the basis functions are disjoint and using the same arguments as above, we can show that $- \mathbb{L}_{\tau}$ is non-positive for any $u$ in the $N$--dimensional subspace of $H^1_0(D)$ spanned by the $u_j$'s. This proves that claim since $N$ is arbitrary. The same result can be proven for $\mathbb{L}_k$ exactly in a similar way for the case when $\lambda\in (0,1)$, $\eta_{min}> 0$, and $\lambda n_{min}-1>0$.
\end{proof}

We have shown five important properties that will compile to imply the existence of transmission eigenvalues. This requires appealing to the following theorem first introduced in \cite{C and K} to study anisotropic transmission eigenvalue problems.

\begin{theorem}\label{eiglemma}
Assume that we have  $\mathbb{L}_k:H_0^1(D)\longrightarrow H_0^1(D)$ that satisfies  
\begin{enumerate}
\item $\mathbb{L}_k$ is self-adjoint and it depends on $k>0$ continuously
\item $\pm\mathbb{L}_0$ is coercive
\item $\mathbb{L}_k-\mathbb{L}_0$ is compact
\item There exists $\alpha>0$ such that $\mathbb{L}_{\alpha}$ is a positive operator
\item There exists $\beta>0$ such that  $\mathbb{L}_{\beta}$ is non-positive on an $m$ dimensional subspace
\end{enumerate}
Then there exists $m$ values $ k_j \in (\alpha , \beta)$ such that  $\mathbb{L}_{k_j}$ has a non-trivial subspace.
\end{theorem}
\begin{proof}
The proof of this result can be found in   \cite{C and K}  Theorem 2.6.
\end{proof}

By the above result as well as the analysis presented in this section we have  the main result of the paper. This gives that there exists infinitely many transmission eigenvalues. 
 
\begin{theorem}\label{exist-thm}
Assume either $\lambda \in (1,\infty)$, $ \eta_{max} <0$, and $\lambda n_{max}-1<0$, or $\lambda\in (0,1)$, $\eta_{min}> 0$ and $\lambda n_{min}-1>0$ respectively, then
there exists infinitely many real transmission eigenvalues $k>0$.
\end{theorem}
\begin{proof}
The proof follows directly by applying Theorem \ref{eiglemma} where we have proven that our operator satisfies the assumptions in the previous results. 
\end{proof}

We have shown the existence of real transmission eigenvalues and we now wish to study how they depend on the parameters $\lambda$, $n$, and $\eta$. We will show monotonicity results for the first transmission eigenvalue with respect to the parameters $n$ and $\eta$. We have two different results with respect to $n$ and $\eta$. The first result shows that the first eigenvalue is an increasing function when $\lambda \in (1,\infty)$, $ \eta_{max} <0$, and $\lambda n_{max}-1<0$. Then we show that the first eigenvalue is a decreasing function when $\lambda\in (0,1)$, $\eta_{min}> 0$, and $\lambda n_{min}-1>0$. 

\begin{theorem}\label{mono1}
Assume that the parameters satisfy $\lambda \in (1,\infty)$, $ \eta_{max} <0$, and $\lambda n_{max}-1<0$. 
Therefore, we have that:
\begin{enumerate} 
\item If $n_1\leq n_2$ such that  $\lambda n_j-1<0$, then  $k_1(n_1)\leq k_1(n_2)$.
\item If  $\eta_1\leq \eta_{2}$ such that  $\eta_j<0$, then $k_1(\eta_1)\leq k_1(\eta_2)$.
\end{enumerate}
Here $k_1$ corresponds to the first transmission eigenvalue. 
\end{theorem}
\begin{proof}
Here, we will prove part (1) for the theorem and part (2) can be handled in a similar manner. To this end, notice that if $n_1\leq n_2$, then we have $(1- \lambda n_2)\leq  (1-\lambda n_1).$ Assume that $\lambda \in (1,\infty)$, $ \eta_{max} <0$, and $\lambda n_{2}-1<0$, and that $v_2$ and $w_2$ are the transmission eigenfunctions corresponding to the transmission eigenvalue $k_2=k_1(n_2,\lambda, \eta)$. Therefore, from \eqref{26} we obtain that
$$ \int_D|\nabla u_2|^2-k_2^2|u_2|^2 \, \text{d}x+\int_D(\lambda-1)|\nabla w_2|^2+ k_2^2(1-\lambda n_2)|w_2|^2 \, \text{d}x-\int_{\partial D}\eta|w_2|^2 \, \text{d}s=0$$
where $u_2=w_2 - v_2 \in H^1_0(D)$.

Now, we have the existence of $v\in H^1(D)$ that solves the variational problem \eqref{UtoV} with $u=u_2$, $n=n_1$, and $k=k_2$. Then, we can define $w=v+u_2$. 
By rearranging the variational form in \eqref{UtoV} and using the definition  $w=v+u_2$ we have that 
\begin{align} \label{ww2}
\int_D (1-\lambda) \nabla w\cdot \nabla \overline{\phi}&-k_2^2(1-\lambda n_1)w\overline{\phi}  \, \text{d}x+\int_{\partial D}\eta w \overline{\phi}  \, \text{d}s\nonumber \\
&=\int_D \nabla u_2 \cdot \nabla \overline{\phi} -k^2_2 u_2\overline{\phi}  \, \text{d}x \nonumber \\
 &=\int_D (1-\lambda) \nabla w_2\cdot \nabla \overline{\phi}-k_2^2(1-\lambda n_2)w_2\overline{\phi}  \, \text{d}x+\int_{\partial D}\eta w_2 \overline{\phi}  \, \text{d}s.
\end{align}
Letting $\phi=w$ in \eqref{ww2} and using the Cauchy-Schwartz inequality as in the proof of Theorem \ref{T4.4}, we have that
\begin{align*} 
\int_{D}(\lambda-1)|\nabla w|^2 &-k_2^2(\lambda n_1-1)|w|^2 \, \text{d}x-\int_{\partial D}\eta|w|^2 \, \text{d}s \\
&\leq\int_{D}(\lambda-1)|\nabla w_2|^2 -k_2^2(\lambda n_2-1)|w_2|^2 \, \text{d}x-\int_{\partial D}\eta|w_2|^2 \, \text{d}s.
\end{align*}
We denote the operator $-\mathbb{L}_\tau$ as the operator with $n=n_1$. By appealing to the calculations in Theorem \ref{T4.3} and the above inequality we have that
\begin{align*}
    -(\mathbb{L}_{k_{2}}u_2,u_2)_{H^1(D)}&=-\int_D\nabla v\cdot \nabla \overline{u_2}-k_2^2 v\overline{u_2}  \, \text{d}x\\ 
    &=\int_{D} |\nabla u_2|^2- k_2^2 |u_2|^2 \, \text{d}x+\int_{D}(\lambda-1)|\nabla w|^2 -k_2^2(\lambda n_1-1)|w|^2 \, \text{d}x\\ 
    &\hspace{7cm}-\int_{\partial D}\eta|w|^2 \, \text{d}s\\
    &\leq \int_{D} |\nabla u_2|^2- k_2^2 |u_2|^2 \, \text{d}x+\int_{D}(\lambda-1)|\nabla w_2|^2 -k_2^2(\lambda n_2-1)|w_2|^2 \, \text{d}x\\
    &\hspace{7cm}-\int_{\partial D}\eta|w_2|^2 \, \text{d}s\\
    &=0.
\end{align*}
Since $-\mathbb{L}_{k_{2}}$ is nonpositive on the subspace spanned by $u_2$ we can conclude that there is an eigenvalue corresponding to $n_1$ in $(0,k_1(n_2)]$. Therefore, the first transmission eigenvalue $k_1(n_1)$ must satisfy that $k_1(n_1)\in (0,k_1(n_2)]$, proving the claim.
\end{proof}

Next, we have a similar monotonicity result with respect to the assumptions on the coefficients that  $\lambda\in (0,1)$, $\eta_{min}> 0$, and $\lambda n_{min}-1>0$. Since the proof is similar to what is presented in Theorem \ref{mono1} we omit the proof to avoid repetition. 

\begin{theorem}\label{mono2}
 Assume that parameters satisfy $\lambda\in (0,1)$, $\eta_{min}> 0$, and $\lambda n_{min}-1>0$. 
 Therefore, we have that:
\begin{enumerate} 
\item If $n_1\leq n_2$ such that $\lambda n_j-1>0$, then $k_1(n_2)\leq k_1(n_1)$, 
\item If $\eta_1\leq \eta_{2}$ such that  $\eta_j>0$, then $k_1(\eta_2)\leq  k_1(\eta_1)$.
\end{enumerate} 
Here $k_1$ corresponds to the first transmission eigenvalue. 
\end{theorem}

From Theorems \ref{mono1} and \ref{mono2} this we can see that the first transmission eigenvalue depend monotonically on some of the material parameters $n$ and $\eta$. Notice, that we are unable to prove a similar monotonicity result with respect to $\lambda$ due to showing up in the variational definition of $\mathbb{L}_k$ in different terms with different signs. We will present some numerics for the monotonicity with respect to $\lambda$ in Section \ref{numerics}.

\section{Convergence as the conductivity $\lambda$ goes to $1$}\label{limit}

In this section, we study the convergence of the transmission eigenvalues in the sense of whether or not we have that $k(\lambda) \longrightarrow k(1)$ as $\lambda\longrightarrow1$ where $k(1)$ is the transmission eigenvalue corresponding to $\lambda=1$. Throughout this section, we will assume that the transmission eigenvalues $k(\lambda)=k_{\lambda}\in \mathbb{R}_+$ form a bounded set as $\lambda\to 1$. From this we have that the set will have a limit point as $\lambda$ tends to one. For the eigenfunctions $v_\lambda$ and $w_\lambda$, we may assume that they are normalized in $H^1(D)$ such that 
$$\|v_\lambda\|^2_{H^1(D)}+\|w_\lambda\|^2_{H^1(D)}=1$$
 for any $\lambda \in (0,1)\cup (1,\infty)$. {As a result}, we have that $(k_{\lambda},v_\lambda,w_\lambda)\in \mathbb{R}_+\times H^1(D)\times H^1(D)$ are bounded, so there exists $(\kappa , \hat{v},\hat{w})\in \mathbb{R}_+\times H^1(D)\times H^1(D)$ such that 
$$k_\lambda\longrightarrow \kappa$$
as well as 
$$ w_\lambda\weakc \hat{w} \hspace{.3cm}\text{and}\hspace{.3cm}v_\lambda\weakc \hat{v} \hspace{.3cm}\text{in}\hspace{.3cm}H^1(D)\hspace{.3cm}\text{as}\hspace{.3cm} \lambda\longrightarrow 1.$$

Now, our task is to show that the limits $ \hat{w}$ and $ \hat{v}$ satisfies the transmission eigenvalue problem when we let $\lambda=1$ with eigenvalue $\kappa$. 
To this end, we begin by showing that the difference of the eigenfunctions $u_\lambda =w_\lambda - v_\lambda$ is bounded with respect to $\lambda$ in the $H^2(D)$--norm. To this end, by \eqref{eigproblem1} we have that
$$\Delta u_{\lambda}+k_{\lambda}^2nu_{\lambda}=  - k_{\lambda}^2(n-1)v_{\lambda} \quad \text{ in } \,\, D.$$
Notice, the fact that $u_\lambda \in H^2(D)\cap H^1_0(D)$ is  given by appealing to standard elliptic regularity results. Observe that
$\|\Delta \cdot\|_{L^2(D)}$ is equivalent to $ \|\cdot\|_{H^2(D)}$ in  
 $H^2(D)\cap H^1_0(D)$ (see for e.g. \cite{salsa}). Therefore, we can bound the  $H^2(D)$--norm of $u_{\lambda}$ using the above equation such that 
 $$\|u_{\lambda}\|_{H^2(D)}\leq C \| \Delta u_\lambda \|_{L^2(D)} \leq C\big\{\|u_{\lambda}\|_{L^2(D)}+\|v_{\lambda}\|_{L^2(D)} \big\}.$$
 Notice, we have used the fact that $n \in L^{\infty}(D)$ and that $k_\lambda$ is bounded with respect to $\lambda$. This implies that, $u_\lambda$ is bounded in $H^2(D)\cap H^1_0(D)$ i.e. 
 $$u_\lambda \weakc \hat{u}= \hat{w} -  \hat{v} \quad \text{in } \quad H^2(D)\cap H^1_0(D) \quad \text{as  } \quad \lambda \to 1.$$

We want to determine which boundary value problem the functions $\hat{u}$ and $ \hat{v}$ satisfy. To this end, we take $\phi\in H^1(D)$ and integrate over the region $D$ to obtain 
$$\int_D(\Delta u_\lambda+k^2_\lambda nu_\lambda)\overline{\phi} \, \text{d}x=-k_\lambda^2 \int_D(n-1)v_\lambda \overline{\phi} \, \text{d}x.$$
Notice, that since $k^2_\lambda\to \kappa^2$ as well as $v_\lambda \to \hat{v}$ in $L^2(D)$ and $u_\lambda \weakc \hat{u}$ in $H^2(D)\cap H^1_0(D)$ as  
$\lambda\to 1$ we have that 
$$\int_D\overline{\phi}\Big[\Delta \hat{u}+\kappa^2n\hat{u}+ \kappa^2(n-1) \hat{v}\Big] \, \text{d}x=0 \quad \text{ for all } \quad \phi \in H^1(D).$$
This implies that  
$$\Delta \hat{u}+\kappa^2n\hat{u}=- \kappa^2(n-1) \hat{v} \quad \text{in }\,\, D.$$
Using a similar argument, we have that 
$$ \Delta \hat{v} +\kappa^2 \hat{v} = 0 \quad \text{ in } \,\, D.$$
Notice, that $u_\lambda|_{\partial D} = 0$ and by the Trace Theorem we have that 
$$\partial_{\nu}u_\lambda|_{\partial D}\in H^{1/2}(\partial D), \quad v_\lambda|_{\partial D}\in H^{1/2}(\partial D), \quad \text{and} \,\,\,\partial_{\nu}v_\lambda|_{\partial D}\in H^{-1/2}(\partial D)$$
are bounded. This implies that the above boundary values weakly converge to the corresponding boundary values for the weak limits. Now, multiplying by $\phi \in H^{1/2}(\partial D)$ and integrating over $\partial D$ in equation \eqref{eigproblem4} we have that
$$\int_{\partial D}  {\phi} \big[ \lambda \partial_{\nu}u_\lambda - \eta v_\lambda \big]  \, \text{d}s = (1-\lambda) \int_{\partial D}{\phi} \partial_{\nu}v_\lambda \, \text{d}s.$$
We can then estimate
\begin{align*}
    \Big|\int_{\partial D}  {\phi} \big[ \lambda \partial_{\nu}u_\lambda - \eta v_\lambda \big]  \, \text{d}s\Big|&\leq | 1-\lambda| \int_{\partial D}|{\phi} \partial_{\nu}v_\lambda | \, \text{d}s\\
    &\leq  |1-\lambda| \|\partial_{\nu}v_\lambda\|_{H^{-1/2}(\partial D)}\|\phi \|_{H^{1/2}(\partial D)}\\
    &\leq C|1-\lambda| \big\{\|v_\lambda\|_{H^1(D)}+\|\Delta v_\lambda\|_{L^2(D)} \big\}\| \phi \|_{H^{1/2}(\partial D)}.
\end{align*}
Notice, that the quantity  
$$\|v_\lambda\|_{H^1(D)}+\|\Delta v_\lambda\|_{L^2(D)}$$ 
is bounded due to the normalization and the  fact that $v_\lambda$ satisfies the Helmholtz equation in $D$. As we let $\lambda\to 1$, we have that 
$$ \int_{\partial D}{\phi} \big[ \partial_{\nu} \hat{u} -\eta  \hat{v} \big] \, \text{d}s=0 \quad \text{ for all } \quad \phi \in H^{1/2}(\partial D).$$
We can conclude that 
$$ \partial_{\nu}\hat{u}= \eta \hat{v} \quad \text{ on } \,\, \partial D.$$ 
Which gives the boundary value problem for the limits. 

Next, we show that as $\lambda\longrightarrow 1$ we have that $u_\lambda\longrightarrow \hat{u}$ in $H^2(D)\cap H^1_0(D)$. From the above analysis, we have obtained that  
\begin{align}\label{29}
\Delta u_\lambda+k_\lambda^2nu_\lambda= - k_\lambda^2(n-1)v_\lambda   \quad \text{ and } \quad   \Delta {v}_\lambda +k_\lambda^2 {v}_\lambda = 0  \hspace{.2cm}& \text{in} \hspace{.2cm}  D\\
 \lambda\partial_{\nu}u_\lambda=(1-\lambda)\partial_{\nu}v_\lambda+\eta v_\lambda \hspace{.2cm}& \text{on} \hspace{.2cm}  \partial D\label{30}
\end{align}
as well as
\begin{align}\label{31}
 \Delta \hat{u}+\kappa^2n\hat{u}=-\kappa^2(n-1) \hat{v} \quad \text{ and } \quad   \Delta \hat{v} +\kappa^2 \hat{v} = 0   \hspace{.2cm}& \text{in} \hspace{.2cm}  D\\
 \partial_{\nu}\hat{u}=\eta  \hat{v}  \hspace{.2cm}& \text{on} \hspace{.2cm}  \partial D\label{32}. 
\end{align}
Notice, that \eqref{31}--\eqref{32} is the transmission eigenvalue problem for $\lambda=1$ as studied in \cite{te-cbc}. This analysis implies that provided that the weak limits are non-trivial as $\lambda \to 1$ we have that $k_\lambda$ converges to the transmission eigenvalue for $\lambda=1$. In order to prove that the weak limits $ \hat{u}$ and $\hat{v}$ are non-trivial we need the following results.

\begin{theorem}
Assume that the coefficients satisfy the assumptions of Theorem \ref{exist-thm} and $k_{\lambda}\in \mathbb{R}_+$ forms a bounded set as $\lambda\longrightarrow 1$. Then $u_\lambda\longrightarrow \hat{u}$ in $H^2(D)\cap H^1_0(D)$ as $\lambda\longrightarrow 1.$
\end{theorem}
\begin{proof}

We subtract \eqref{29} from \eqref{31} to get the following
$$\Delta(u_\lambda-\hat{u})=-k_\lambda^2n(u_\lambda-\hat{u})+n\hat{u}(k_\lambda^2-\kappa^2)+(1-n) \Big(k_\lambda^2(v_\lambda-\hat{v})+\hat{v}(k_\lambda^2-\kappa^2) \Big).$$
Recall, that $u_\lambda$ and $\hat{u} \in H^2(D)\cap H^1_0(D)$. Therefore, by taking $L^2(D)$ norm on both sides we obtain the estimate
$$\|\Delta (u_\lambda-\hat{u}) \|_{L^2(D)}\leq C\big\{ \|u_\lambda-\hat{u}\|_{L^2(D)}+|k_\lambda^2-\kappa^2|+\|v_\lambda-\hat{v}\|_{L^2(D)} \big\}.$$
Where we have used the triangle inequality and that $n$ and $k_\lambda^2$ are both bounded with respect to $\lambda$. Again, using the fact that $\|\Delta \cdot\|_{L^2(D)}$ is equivalent to $\|\cdot\|_{H^2(D)}$ in $H^2(D)\cap H^1_0(D)$ gives us 
$$\|u_\lambda-\hat{u}\|_{H^2(D)}\leq C\big\{ \|u_\lambda-\hat{u}\|_{L^2(D)}+|k_\lambda^2-\kappa^2|+\|v_\lambda-\hat{v}\|_{L^2(D)} \big\}.$$ 
 The above inequality implies that  $u_\lambda\longrightarrow \hat{u}$ in $H^2(D)\cap H^1_0(D)$ as  $\lambda \longrightarrow 1$ by the compact embedding of $H^1(D)$ into $L^2(D)$. 
\end{proof} 

We will now use the above convergence result to prove that $\hat{u}\in H^2(D)\cap H^1_0(D)$ is non-trivial under some further assumptions. 

\begin{theorem}\label{5.2}
Assume that the coefficients satisfy the assumptions of Theorem \ref{exist-thm} as well as $n-1\neq 0$ a.e. in $D$ and $\partial_{\nu}v_{\lambda}$ is bounded in $L^2(\partial D)$. Then $\hat{u}$ is non-trivial. 
\end{theorem}
\begin{proof}
For contradiction, assume $\hat{u} =0$. Now, recall that we have 
$$\Delta u_{\lambda}+k_{\lambda}^2nu_{\lambda}=- k_{\lambda}^2(n-1)v_{\lambda}$$
 and by the convergence as $\lambda\to 1$ we have that 
 $$0=- \kappa^2(n-1)\hat{v} \quad \text{  in $D$}.$$ 
 Now, as we have that $k_{\lambda}^2$ is bounded below as a consequence of Theorem \ref{T4.3} and $n-1\neq 0$, this implies that $\hat{v}=0.$ Thus, we have that $v_\lambda \rightharpoonup 0$ in $H^1(D)$ and by compact embedding $v_{\lambda}\to 0$ in $L^2(D).$ We will now show that $\nabla v_{\lambda}$ strongly converges to the zero vector. Recall, that the function $v_{\lambda}\in H^1(D)$ satisfies Helmholtz equation, i.e. $\Delta v_{\lambda}+k_{\lambda}^2v_{\lambda}=0$ in $D$. Using Green's 1st Theorem gives
$$\int_{\partial D}\overline{\phi}\partial_{\nu}v_{\lambda} \, \text{d}s=\int_{D}\overline{\phi}\Delta v_{\lambda}+\nabla v_{\lambda}\cdot \nabla\overline{\phi} \, \text{d}x\hspace{.5cm}\text{for}\hspace{.5cm}\phi\in H^1(D).$$
Letting $\phi=v_{\lambda}$ in the above equality gives that 
$$\int_{\partial D} \overline{v_{\lambda}} \partial_{\nu}v_{\lambda} \, \text{d}s=\int_{D} \overline{v_{\lambda}} \Delta v_{\lambda}+ |\nabla v_{\lambda}|^2 \, \text{d}x=-\int_{D}k_{\lambda}^2|v_{\lambda}|^2 \, \text{d}x+\int_{D}|\nabla v_{\lambda}|^2  \, \text{d}x.$$
Observe that 
$$\int_{D}|\nabla v_{\lambda}|^2  \, \text{d}x=\int_{D}k_{\lambda}^2|v_{\lambda}|^2 \, \text{d}x+\int_{\partial D}\overline{v_{\lambda}} \partial_{\nu}v_{\lambda} \, \text{d}s.$$
Using the Cauchy-Schwarz inequality we get that 
$$\|\nabla v_{\lambda}\|^2_{L^2(D)}\leq \|\partial_{\nu}v_{\lambda}\|_{L^2(\partial D)} \|v_{\lambda}\|_{L^2(\partial D)}+k^2_{\lambda}\|v_{\lambda}\|^2_{L^2(D)}$$
which implies that 
$$\|\nabla v_{\lambda}\|^2_{L^2(D)}\leq C \big\{  \|v_{\lambda}\|_{L^2(\partial D)}+ \|v_{\lambda}\|^2_{L^2(D)} \big\}$$
since we have assumed that $\|\partial_{\nu}v_{\lambda}\|_{L^2(\partial D)}$ and $k_\lambda$ are bounded. 
By the compact embedding of $H^{1/2}(\partial D)$ into $L^2(\partial D)$ we have that
 $$v_{\lambda} \rightharpoonup 0 \, \text{ in } \, H^{1/2}(\partial D) \quad \text{ implies } \quad v_{\lambda}\to 0 \, \text{ in } \, L^2(\partial D).$$
Using the fact that  $v_{\lambda}\to 0$ in $L^2(D)$ we can conclude that $v_\lambda \to 0$  in $H^1(D)$ by the above inequality. Therefore, we have that both $u_\lambda$ and  $v_\lambda$ converge to zero in $H^1(D)$. Now, because we have that $u_\lambda=w_\lambda-v_\lambda$ we obtain that $w_\lambda$ converges to zero in $H^1(D)$. This contradicts the normalization 
 $$\|v_\lambda\|^2_{H^1(D)}+\|w_\lambda\|^2_{H^1(D)}=1$$ 
 proving the claim.
 \end{proof}

Now, putting everything together, we are able to state the main result of this section. Here, we have that as $\lambda \to 1$ the transmission eigenvalues will have a limit that corresponds to standard eigenvalue problem when $\lambda = 1$ under some assumptions. 

\begin{theorem}\label{conv}
Assume that the coefficients satisfy the assumptions of Theorem \ref{exist-thm} as well as $n-1\neq 0$ a.e. in $D$ and $\partial_{\nu}v_{\lambda}$ is bounded in $L^2(\partial D)$. Then, we have that $k_{\lambda} \to k(1)$ as $\lambda\to 1$ where $k(1)$ is a  transmission eigenvalue corresponding to $\lambda =1$. 
\end{theorem}
\begin{proof}
The proof is a simple consequence of the analysis presented in this section. 
 \end{proof}

{We note that since $k_{\lambda}$ and $k(1)$ are chosen arbitrarily, the above result holds for all transmission eigenvalues, without assuming their exact position in the real spectrum. This means that for the ordered subsequence of real eigenvalues we have  $k_{\lambda,j} \to k_j(1)$ for all $j=1,2,\dots$, where $k_{\lambda,1}$ is the first, $k_{\lambda,2}$ the second etc.} 

We have shown the monotonicity with respect to $n$ and $\eta$ where as now we have an understanding of the limiting process as $\lambda \to 1$. In the case of inverse problems, it is very useful to understand how the eigenvalues of a differential operator depend on the coefficients. From an application perspective, this implies that the transmission eigenvalues can be used as a target signature to determine information about the scatterer since the eigenvalues can be recovered from the scattering data. 

\section{Numerical Validation}\label{numerics}
In this section, we provide some numerical examples that validate the theoretical results from the previous sections. First, we will {give} some numerical examples of the convergence $k(\lambda)\longrightarrow k(1)$ as $\lambda \to 1$ in Theorem \ref{conv} for the unit ball with constant coefficients. Here we will consider the convergence and estimate the rate of convergence for the case when $\lambda \in (0,1)$ and $\lambda \in (1,\infty)$. Then, we will provide some examples {for} the monotonicity of the eigenvalues with respect to the parameters $n$ and $\eta$ given in Theorems \ref{mono1} and \ref{mono2}. Lastly, we will also report the transmission eigenvalues for other shapes using boundary integral equations.

\subsection{Validation on the unit disk for the Convergence of $\lambda$}
Here, we consider the convergence of the $k_\lambda$ as $\lambda \to 1^{\pm}$. For this we will assume that $D=B(0,1) \subset \R^2$ (i.e. the unit disk centered at the origin) and that coefficients $n,\eta,$ and $ \lambda$ are all constants. Under these assumptions, we recall that the transmission eigenvalue problem is given by 
\begin{align}
\Delta w+k^2 n w = 0  \quad \text{ and}  \quad \Delta v+k^2v = 0 \hspace{.2cm}& \text{in} \hspace{.2cm} B(0,1) \label{33} \\
w=v  \quad \text{ and}  \quad  \lambda \partial_{r}w=\partial_{r}v+\eta v \hspace{.2cm}& \text{on} \hspace{.2cm}  \partial B(0,1). \label{36}
\end{align}
Motivated by separation of variables, we try to find eigenfunctions of the form 
$$w(r,\theta)=w_m(r)\text{e}^{\text{i}m\theta} \quad \text{ and } \quad v(r,\theta)=v_m(r)\text{e}^{\text{i}m\theta}$$
 where $m\in \mathbb{Z}$. From this we obtain that {$w_m(r) = \alpha_m J_m(k\sqrt{n}r)$ and $v_m(r)=\beta_m J_m(kr)$} where both {$\alpha_m$} and {$\beta_m$} are constants. Therefore, applying the boundary conditions at $r=1$ gives that the transmission eigenvalues are given by the roots of $d_m(k),$ {defined} by 
\begin{equation}\label{41}
 \hspace{2cm}d_m(k):=\text{det}\begin{pmatrix}
J_m(k\sqrt{n}) &-J_m(k)\\
\lambda J^{'}_{m}(k\sqrt{n})k\sqrt{n} & -\big(kJ^{'}_{m}(k)+\eta J_m(k) \big)
\end{pmatrix}.
\end{equation}
Here we let $J_m(t)$ denote the Bessel functions of the first kind of order $m$. 
\begin{figure}[ht]
\centering 
\includegraphics[width=.65\linewidth]{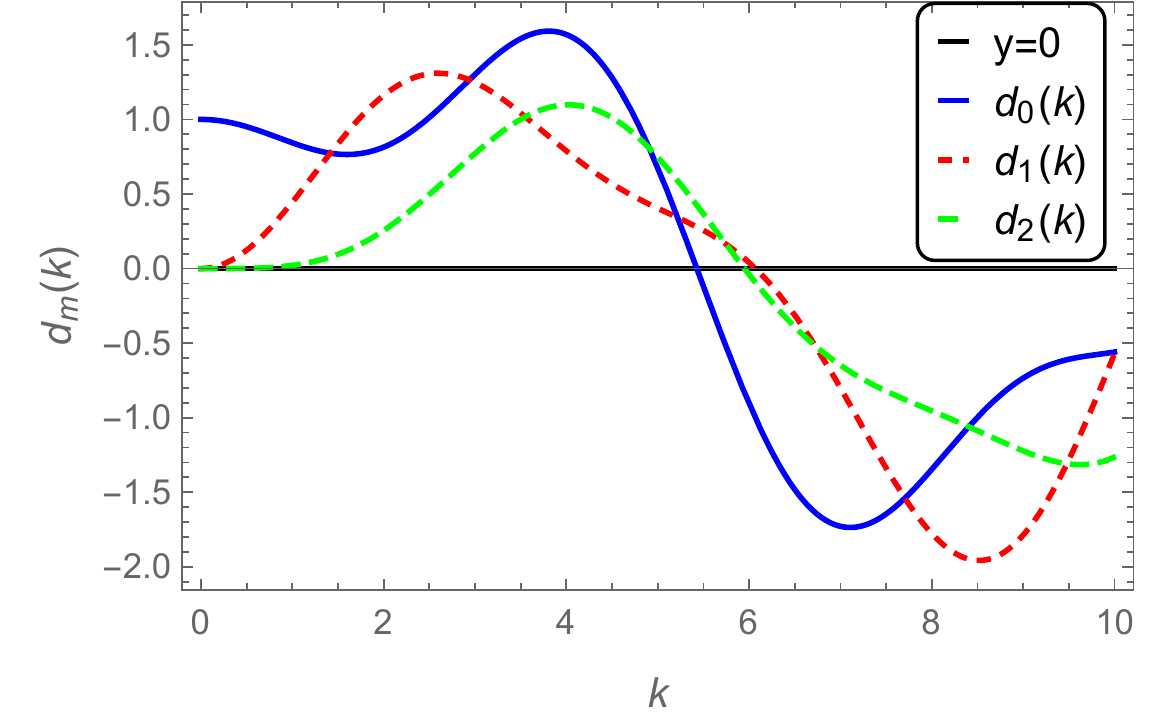}
\caption{The plots of the {determinant} function $d_m(k)$ for $m=0,1,2$. Here the parameters are given by $n=1/6$, $\lambda = 5$, and $\eta=-1$. }
\label{plot-of-func}
\end{figure}

Letting $k_\lambda$ be the root(s) of $d_m(k)$, we can see that the eigenfunctions are given by 
$$w_{\lambda}(r,\theta)=J_m(k_\lambda)J_m(k_{\lambda}\sqrt{n}r)\text{e}^{\text{i}m\theta} \quad \text{ and } \quad v_{\lambda}(r,\theta)=J_m(k_{\lambda}\sqrt{n})J_m(k_{\lambda}r)\text{e}^{\text{i}m\theta} .$$ 
One can easily check that such forms satisfy the boundary conditions and also that if $k_\lambda$ forms a bounded set then
$$\| \partial_{r}v_\lambda (1,\theta)\|_{L^2(0,2 \pi)} \quad \text{ is bounded with respect to $\lambda$. }$$

We note that the position of each eigenvalue on the spectrum, is not directly associated with the order $m$ of the determinant $d_m(k)$, of which is a root. This means for e.g., that the the lowest eigenvalue $k_1$ can be the first root of $d_1(k)$ (or of other order) and not $d_0(k)$. As a result, in the examples following, we calculate the roots and sort them in ascending order. 
 
Now, we wish to provide some numerical validation of Theorem \ref{conv}. First, we give some examples when we let $\lambda$ approach $1$ from below then we check the case when $\lambda$ approach $1$ from above. The examples are given by considering the first three transmission eigenvalues, as roots of $d_m(k)$, for $m=0,1,2,\ldots$.Therefore,  we have that the limiting value as $\lambda$ tends to 1 of the transmission eigenvalues, are the corresponding roots for $\lambda =1$. When $\lambda=1,$ $\eta=1$ and $n=4$ we have that $k_1(1)=2.7741$, $k_2(1)=3.2908$,  and $k_3(1)=3.3122$ are the first three limiting transmission eigenvalues, coming from $d_1(k),\ d_0(k)$, and $d_2(k)$ respectively. From this, we show that numerically $k_j(\lambda)\longrightarrow k_j(1)$ as $\lambda\longrightarrow 1^-$ for $j=1,2,3$ and the results are presented in Table \ref{first table of k's}. We also check the estimated order of convergence (EOC) which is given by 
$$\text{EOC}=\log(\epsilon_{\lambda_p}/\epsilon_{\lambda_{p+1}})/\log(2)\hspace{.2cm}\text{where}\hspace{.2cm}\epsilon_{\lambda_p}=|k_j(\lambda_p)-k_j(1)| \hspace{.2cm}\text{for}\hspace{.2cm}j=1,2,3$$
$$\text{and}\hspace{.2cm}\lambda_{p}=1-\frac{1}{2^p}\hspace{.2cm}\text{for}\hspace{.2cm}p=1,2,3\dots$$
where our calculations suggest first order convergence as $\lambda\longrightarrow 1^-$.  Also, notice that in Table \ref{first table of k's} the eigenvalues seem to be monotone with respect to $\lambda$. We see that $k_1(\lambda)$ is descending $k_2(\lambda)$ and $k_3(\lambda)$ are ascending with respect to $\lambda$.

\begin{table}[!ht]
  \begin{center}
    \begin{tabular}{l|c |c |r r} 
      $\lambda$ & $k_1(\lambda)$ \hspace{.5cm} EOC& $k_2(\lambda)$ \hspace{.5cm} EOC&$k_3(\lambda)$ \hspace{.5cm} EOC \\
      \hline
    $1 -1/2$ & 3.0394 \hspace{.8cm} N/A & 3.0561 \hspace{.8cm} N/A & 3.2494 \hspace{.8cm} N/A\\
    $1 -1/4$ & 2.8388 \hspace{.5cm}  2.0346 & 3.1970 \hspace{.5cm} 1.3241 & 3.2942 \hspace{.5cm} 1.8057\\
    $ 1 -1/8$ & 2.7990 \hspace{.5cm} 1.3774 & 3.2509 \hspace{.5cm} 1.2313 & 3.3048 \hspace{.5cm} 1.2831\\
    $ 1 -1/16$ & 2.7853 \hspace{.5cm} 1.1590 & 3.2723 \hspace{.5cm} 1.1092 & 3.3088 \hspace{.5cm} 1.1223\\
    $ 1 -1/32 $& 2.7794  \hspace{.5cm} 1.0770 & 3.2819 \hspace{.5cm} 1.0516 & 3.3106 \hspace{.5cm} 1.0476\\
    $ 1 -1/64 $&  2.7767 \hspace{.5cm} 1.0415 &  3.2864 \hspace{.5cm} 1.0212  & 3.3114 \hspace{.5cm} 1.0177\\
    $ 1 -1/128$ & 2.7754 \hspace{.5cm} 1.0283 & 3.2886 \hspace{.5cm} 1.0066 & 3.3118 \hspace{.5cm} 0.9823 \\
    $  1 -1/256$ &  2.7747 \hspace{.5cm} 1.0465&
    3.2897 \hspace{.5cm} 0.9934 & 3.3120 \hspace{.5cm} 0.9652\\
    $1 -1/512$ & 2.7744 \hspace{.5cm} 1.0728 & 3.2902 \hspace{.5cm} 0.9740 & 3.3121 \hspace{.5cm} 0.9329 \\
    $ 1 -1/1024$ & 2.7742 \hspace{.5cm}  1.1575 & 3.2905 \hspace{.5cm} 0.9494 & 3.3121 \hspace{.5cm} 0.8745
    \end{tabular}
  \end{center}
  \caption{Convergence of the transmission eigenvalues as $\lambda\longrightarrow 1^-$ for  $n=4$ and $\eta=1$. Here, the limiting values are $k_1(1)=2.7741,$ $k_2(1)=3.2908$, and $k_3(1)=3.3122$.}
\label{first table of k's}
\end{table}

 We now give a numerical example of the convergences when $\lambda\in (1,\infty)$. It is important to remember that for this case, we have that $\eta_{max} <0$ and $\lambda n_{max}-1<0$ as $\lambda\longrightarrow 1^+$. We again compute the EOC with  
$$\lambda_{p}=1+\frac{1}{2^p}\hspace{.2cm}\text{for}\hspace{.2cm}p=1,2,3\dots$$
to establish the convergence rate. For Table \ref{third table of k's}, we choose $n=1/3$ and $\eta=-1$ following the assumptions on the coefficients given in Theorem \ref{conv}.  Again, we compute the lowest three roots of $d_m(k)$ for $\lambda_{p}$. We have that the limiting transmission eigenvalues for $n=1/3$ and $\eta=-1$  are given by $k_1(1)=6.9883,$ $k_2(1)=7.0107$, and $k_3(1)=7.9523$, being the first roots of $d_0(k),\ d_2(k)$, and $d_1(k)$ respectively. 

\begin{table}[!ht]
  \begin{center}
       \begin{tabular}{l|c |c |r r} 
      $\lambda$ & $k_1(\lambda)$ \hspace{.5cm} EOC& $k_2(\lambda)$ \hspace{.5cm} EOC&$k_3(\lambda)$ \hspace{.5cm} EOC \\
      \hline
      $1 +1/2$ & 7.1094 \hspace{.8cm} N/A & 7.4849 \hspace{.8cm} N/A & 7.6108 \hspace{.8cm} N/A\\
       $1 +1/4$ & 7.0395 \hspace{.5cm} 1.2433 & 7.2250 \hspace{.5cm} 1.1455 & 7.7774 \hspace{.5cm} 0.9655\\
      $ 1 +1/8$ & 7.0121 \hspace{.5cm} 1.1084 & 7.1108 \hspace{.5cm} 1.0984 &  7.8660 \hspace{.5cm} 1.0189\\
       $ 1 +1/16$ & 6.9998 \hspace{.5cm} 1.0527 & 7.0590 \hspace{.5cm} 1.0513 & 7.9097 \hspace{.5cm} 1.0168\\
      $ 1 +1/32 $& 6.9940 \hspace{.5cm} 1.0268 & 7.0344 \hspace{.5cm} 1.0265 &  7.9311 \hspace{.5cm} 1.0085\\
      $ 1 +/64 $& 6.9912 \hspace{.5cm} 1.0181 & 7.0224 \hspace{.5cm} 1.0165 &  7.9417 \hspace{.5cm} 1.0027\\
     $  1 +1/128$ & 6.9898 \hspace{.5cm} 1.0157&
    7.0165 \hspace{.5cm} 1.0124 & 7.9470 \hspace{.5cm} 0.9973\\
       $ 1 +1/256$ & 6.9891 \hspace{.5cm} 1.0106 & 7.0136 \hspace{.5cm} 1.0125 & 7.9496 \hspace{.5cm} 0.9892\\
       $1 +1/512$ &  6.9887 \hspace{.5cm} 1.0431 & 7.0121 \hspace{.5cm} 1.0304 & 7.9509 \hspace{.5cm}  0.9732\\
      $ 1 +1/1024$ & 6.9886 \hspace{.5cm} 1.1375 & 7.0114 \hspace{.5cm} 1.0521 & 7.9516 \hspace{.5cm} 0.9582
    \end{tabular}
  \end{center}
  \caption{Convergence of the transmission eigenvalues as $\lambda\longrightarrow 1^+$ for $n=1/3$ and $\eta=-1$. Here the limiting values are $k_1(1)=6.9883,$ $k_2(1)=7.0107$ and $k_3(1)=7.9523$.}
\label{third table of k's}
\end{table}

We again notice that, in Table \ref{third table of k's} the eigenvalues seem to be monotone with respect to $\lambda$. We see that $k_1(\lambda)$ and $k_2(\lambda)$ are increasing where as $k_3(\lambda)$ is decreasing with respect to $\lambda$. Although, we only showed that there is convergence, we have these numerical examples that seem to suggest monotonicity  of the transmission eigenvalues with respect to the parameter $\lambda$. Here, we conjecture the monotonicity but due the variational form studied in the previous section we are unable to obtain this result theoretically.

\subsection{Monotonicity of $\eta$ and $n$ on the Unit Disk}
Here, we will provide some numerics for the monotonicity with respect to  $\eta$ and $n$ given in Theorems  \ref{mono1} and \ref{mono2}. Just as in the previous section, we will assume that $D$ is the unit disk with constant coefficients.  Therefore, we can again use the fact that $k$ is a transmission eigenvalue provided that it is a root for $d_m(k)$ given by \eqref{41}.

We first consider the monotonicity with respect to the parameter $n$. To this end, recall that $\lambda\in (1,\infty), \eta_{max}<0$, and $\lambda n_{max}-1<0$. Therefore, we fix $\lambda=2$ and $\eta=-3$ and report the transmission eigenvalues $k_j(n)$ for $j=1,2$ corresponding to the lowest two roots of $d_m(k)$, in Table \ref{n1}. 
\begin{table}[!ht]
  \begin{center}
    \begin{tabular}{c|c |c|c|c } 
    $n$ & $1/6$ & $1/5$ & $1/4$ & $1/3$\\
    \hline 
    $k_1(n)$ & 4.8387 & 4.9935 & 5.6504 & 6.5592\\ 
    $k_2(n)$ & 4.8893 & 5.6474 & 6.0112 & 7.3299
    \end{tabular}
  \end{center}
   \caption{Monotonicity with respect to $n$ where $\lambda=2$ and $\eta=-3$ for the unit disk. Here, $k_j$ are the first two transmission eigenvalues. }
\label{n1}
\end{table}

In a similar fashion, we now provide numerical examples for the case when the parameters  $\lambda\in (0,1)$, $\eta_{min}>0$, and $\lambda n_{min}-1>0$ corresponding to Theorem \ref{mono2}. Therefore, we again report the first two roots of the functions $d_m(k)$. In Table \ref{n2}, we fix $\lambda=1/2$ and $\eta=1$ for $k_j(n)$ for $j=1,2$. 
\begin{table}[!ht]
  \begin{center}
    \begin{tabular}{c|c |c|c|c|c} 
         $n$ & $3$& $4$ & $5$ & $6$ & $7$\\
    \hline 
    $k_1(n)$ & 3.9850 & 3.0394 & 2.3699 & 2.0651 & 1.6559\\ 
    $k_2(n)$ & 4.2464  & 3.0561 & 2.5280 & 2.0706 & 1.8761\\
          \end{tabular}
  \end{center}
   \caption{Monotonicity with respect to $n$ where $\lambda=1/2$ and $\eta=1$ for the unit disk. Here, $k_j$ are the first two transmission eigenvalues.}
\label{n2}
\end{table}

 Next, we turn our attention to the monotonicity with respect to $\eta$. We first consider the case where we have $\lambda\in (1,\infty)$, $\eta_{max}<0$, and $\lambda n_{max}-1<0$. Recall, that from Theorem \ref{mono1} we expect that the transmission eigenvalues to be increasing with respect to $\eta$. In Table \ref{eta1}, we fix $\lambda=5$ and $n=1/6$ to compute $k_j(\eta)$ for $j=1,2$ and we can see the monotonicity from the reported values. 
 \begin{table}[!ht]
  \begin{center}
    \begin{tabular}{c|c |c|c|c|c} 
      $\eta$ & $-4$& $-3$ & $-2$ & $-1$ & $-1/2$\\
    \hline 
    $k_1(\eta)$ &  4.7141 & 5.0753  &  5.4263  & 5.4283 & 5.4293\\ 
    $k_2(\eta)$ & 5.4220 & 5.4242& 5.7292 & 5.9486 &6.0176\\
    \end{tabular}
  \end{center}
   \caption{Monotonicity with respect to $\eta$ where $\lambda=5$ and $n=1/6$ for the unit circle. Here, $k_j$ are the first two transmission eigenvalues.}
    \label{eta1}
\end{table}

 Now, we focus on case corresponding to Theorem \ref{mono2} where the transmission eigenvalues are decreasing with respect to the parameter $\eta$. Therefore, we need the assumptions $\lambda\in (0,1)$, $\eta_{min}>0$, and $\lambda n_{min}-1>0$ for the result to hold. In Table \ref{eta2}, we fix $\lambda=1/2$ and $n=3$ for $k_j(n)$ respectively for $j=1,2$. 
\begin{table}[!ht]
  \begin{center}
    \begin{tabular}{c|c |c |c|c|c} 
     $\eta$ & $1$& $2$ & $3$ & $4$ & $5$\\
    \hline 
    $k_1(\eta)$ &   3.9850& 3.6700 &  3.5212 & 2.6262 & 1.6354\\ 
    $k_2(\eta)$ & 4.2464  &  4.0269 & 3.5409 & 3.1242 & 1.9005
    \end{tabular}
  \end{center}
     \caption{ Monotonicity with respect to $\eta$ where $\lambda=1/2$ and $n=3$ for the unit circle. Here, $k_j$ are the first two transmission eigenvalues.}
         \label{eta2}
     \end{table}

\subsection{Numerics via Boundary Integral Equations}
The derivation of the boundary integral equation to solve the problem follows along the same lines as in \cite[Section 3]{te-cbc3} where one uses a single layer ansatz for the functions $w$ and $v$ with unknown densities $\varphi$ and $\psi$ (refer also to \cite{cakonikress} for the original idea). Precisely, we use
\[w(x)=\mathrm{SL}_{k\sqrt{n}}\varphi(x)\qquad\text{and}\qquad v(x)=\mathrm{SL}_{k}\psi(x)\,,\qquad x\in D\,,\]
where we define the single-layer by
\[\mathrm{SL}_{k}\phi(x)=\int_{\partial D}\Phi_k(x,y)\phi(y)\;\mathrm{d}s(y)\,,\qquad  x\in D\]
where 
\[\Phi_k(x,y) = \frac{\text{i}}{4} H^{(1)}_0\big( k|x-y| \big), \quad \text{ when} \,\, x\neq y\]
is the fundamental solution of the Helmholtz equation in two dimensions. Here we let $H^{(1)}_0$ denote the zeroth order first kind Hankel function. On the boundary we have
\[w(x)=\mathrm{S}_{k\sqrt{n}}\varphi(x)\qquad\text{and}\qquad  v(x)=\mathrm{S}_{k}\psi(x)\,,\]
where the boundary operator $\mathrm{S}_{k}$ is given by
\[\mathrm{S}_{k}\phi(x)=\int_{\partial D}\Phi_k(x,y)\phi(y)\;\mathrm{d}s(y)\,,\qquad x\in \partial D\,.\]
Likewise, we obtain 
\[\partial_\nu w(x)=\left(\frac{1}{2}\mathrm{I}+\mathrm{K}^\top_{k\sqrt{n}}\right)\varphi(x)\qquad\text{and}\qquad  \partial_\nu v(x)=\left(\frac{1}{2}\mathrm{I}+\mathrm{K}^\top_{k}\right)\psi(x)\,,\]
where 
\[\mathrm{K}^\top_{k}\phi(x)=\int_{\partial D}\partial_{\nu(x)}\Phi_k(x,y)\phi(y)\;\mathrm{d}s(y)\,,\qquad x\in \partial D\,\]
and $\mathrm{I}$ denotes the identity.
Using the given boundary conditions and assuming that $k$ and $k\sqrt{n}$ are not eigenvalues of $-\Delta$ in $D$ yields
\[\left[\lambda\left(\frac{1}{2}\mathrm{I}+\mathrm{K}^{\top}_{k\sqrt{n}}\right)\mathrm{S}_{k\sqrt{n}}^{-1}-\left(\frac{1}{2}\mathrm{I}+\mathrm{K}^{\top}_{k}\right)\mathrm{S}_{k}^{-1}-\eta \mathrm{I}\right]w=0\,,\]
which is a non-linear eigenvalue problem of the form
\begin{eqnarray}
M(k;n,\eta,\lambda)w=0\,.
\label{sys}
\end{eqnarray}
Here, the parameters $n$, $\eta$, and $\lambda$ are given.
Note that we focus on the transpose of this equation since the boundary integral operator
\[\mathrm{K}_{k}\phi(x)=\int_{\partial D}\partial_{\nu(y)}\Phi_k(x,y)\phi(y)\;\mathrm{d}s(y)\,,\qquad x\in \partial D\]
can be numerically approximated avoiding the singularity (see \cite[Section 4.3]{kleefeldhot} for details and the discretization of the boundary integral operators). Then, the non-linear eigenvalue problem is solved with the Beyn's algorithm (see \cite{beyn2012integral} for a detailed description). This algorithm converts a large scale non-linear eigenvalue problem to a linear eigenvalue problem of smaller size by appealing complex analysis, i.e. contour integrals in the complex plane. The contour we will choose, will be the disk in the complex plane centered at $\mu \in \C$ for a fixed radius $R$. From this, Beyn's algorithm will compute the transmission eigenvalues that lie in the interior of the chosen contour. 
\begin{figure}[!ht]
\centering 
\begin{tikzpicture}
		\begin{axis}[
			height=6 cm, 
			width=6 cm,
			ymin=-1.6,ymax=1.6,
			xmin=-1.6,xmax=1.6,
			xtick={-1.5,-1,-0.5,0,0.5,1,1.5},
			ytick={-1.5,-1,-0.5,0,0.5,1,1.5},
			trig format plots=rad,
			axis equal,
		    tick label style={font=\footnotesize},
		    tick align=center
			]
			\addplot [domain=0:2*pi,samples=200, blue,thick,fill=blue!30]({cos(x)},{1.2*sin(x)});
		\end{axis}
	\end{tikzpicture}\hspace{1cm}
\begin{tikzpicture}
		\begin{axis}[
			height=6 cm, 
			width=6 cm,
			ymin=-1.6,ymax=1.6,
			xmin=-1.6,xmax=1.6,
			xtick={-1.5,-1,-0.5,0,0.5,1,1.5},
			ytick={-1.5,-1,-0.5,0,0.5,1,1.5},
			trig format plots=rad,
			axis equal,
			tick label style={font=\footnotesize},
			tick align=center
			]
			\addplot [domain=0:2*pi,samples=200, blue,thick,fill=blue!30]({0.75*cos(x)+0.3*cos(2*x)},{sin(x)});
		\end{axis}
	\end{tikzpicture}
\caption{Graphical representation for the elliptical and kite-shaped domains considered in this section.}
\label{e.g. scatterer}
\end{figure}
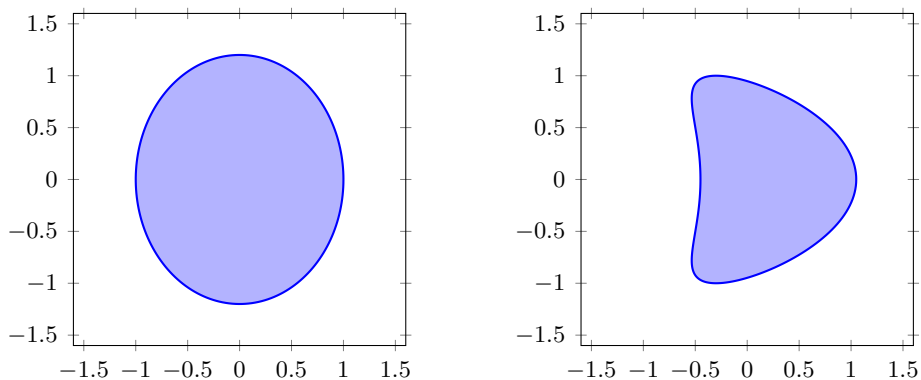

First, we show that we are able to reproduce the values given in Example \ref{ex1} on page \pageref{ex1} for the unit disk using the material parameters $\lambda=2$, $n=4$, $\eta=-1/100$ with the boundary element collocation method. We use $120$ collocation nodes ($40$ pieces) within our algorithm for the discretization of the boundary. For the Beyn algorithm we take the parameters $tol=10^{-4}$, $\ell=20$, and $N=24$ discretization points for the two contour integrals where the contour is a circle with center $\mu$ and radius $R=1/2$. Next, we pick $\mu=3.5$ and obtain the interior transmission eigenvalue $3.4567-0.0000\mathrm{i}$ which agrees with the value reported in Example \ref{ex1} to four digits accuracy. This eigenvalues has multiplicity one (it corresponds to $m=0$). Using $\mu=2.2$ yields the interior transmission eigenvalue $2.1516 - 0.0000\mathrm{i}$ with multiplicity two which is in agreement with the value $2.151602$ obtained from the determinant for $m=4$. Again, we observe that all reported digits are correct. The accuracy does not depend on the multiplicity of the eigenvalue. Finally, we test our boundary element collocation method for a complex-valued interior transmission eigenvalue. Using $\mu=2.2+0.6\mathrm{i}$ yields the simple eigenvalue $2.2032+0.2905\mathrm{i}$ (rounded) which is in agreement to five digits with the value reported in Example \ref{ex1} using the determinant with $m=0$. In sum, this shows that we are able to compute both real and complex-valued interior transmission eigenvalues to high accuracy. It gives us the flexibility to now compute them also for other scatterers as well.

For an ellipse with semi-axis $a=1$ and $b=1.2$ (refer to Figure \ref{e.g. scatterer}) i.e. 
$$\partial D = (\cos(t),1.2 \sin(t)) \quad \text{ for } \,\, t \in [0,2\pi)$$ using $\mu=1/2$ as well as $\mu=3/2$ 
and the same material parameters as before, we obtain the first nine real-valued interior transmission eigenvalues 
\begin{eqnarray*} 
0.0420 \quad 0.6036 \quad 0.7165 \quad 1.0830 \quad 1.1136 \quad 1.5244 \quad 1.5311 \quad 1.9494 \quad 1.9507\,,                                                                                                     
\end{eqnarray*}
where we skipped to report the imaginary eigenvalues.
In comparison, the first nine real-valued interior transmission eigenvalues for the unit disk are
\begin{eqnarray*}
0.0534 \quad 0.7208 \quad 0.7208 \quad 1.2131 \quad 1.2131 \quad 1.6864 \quad 1.6864 \quad 2.1516 \quad 2.1516\,.
\end{eqnarray*}

Next, we compute the interior transmission eigenvalues for the kite-shaped domain (refer to Figure \ref{e.g. scatterer}) using the same parameters as before. Its boundary is given parametrically by 
$$\partial D = (0.75\cos(t)+0.3\cos(2t),\sin(t)) \quad \text{ for } \,\, t \in [0,2\pi)$$ 
(refer to \cite{cakonikress}). We use $\mu=1/2$, $\mu=3/2$ as well as $\mu=5/2$ to 
obtain the first nine real-valued interior transmission eigenvalues
\begin{eqnarray*}
0.0523 \quad 0.6868 \quad 0.8514 \quad 1.3452 \quad 1.4398 \quad 1.6348 \quad 2.0181 \quad 2.1439 \quad 2.3494\,.                                          
\end{eqnarray*}

Now, we consider the ellipse with semi-axis $a=1$ and $b=1.2$ and use the material parameters $\eta=1$ and $n=4$ and vary $\lambda$ such that it approaches one from below. We will validate again numerically Theorem \ref{conv} as it was done for the unit disk in Table \ref{first table of k's}. The results are reported in Table \ref{n1ellipse}. Note that the first three real-valued interior transmission eigenvalues for $\lambda=1$ are $k_1(1)=2.4343$, $k_2(1)=2.6726$, and $k_3(1)=2.8300$ which we obtained using $\mu=5/2$ with $240$ collocation nodes.
\begin{table}[ht]
  \begin{center}
    \begin{tabular}{l|c |c |r r} 
      $\lambda$ & $k_1(\lambda)$ \hspace{.5cm} EOC& $k_2(\lambda)$ \hspace{.5cm} EOC&$k_3(\lambda)$ \hspace{.5cm} EOC \\
      \hline
         $ 1 -1/2$ & 2.5043 \hspace{.8cm} {N/A} & 2.7413 \hspace{.8cm} {N/A} & 2.8777 \hspace{.8cm} {N/A}\\
         $ 1 -1/4$ & 2.4701 \hspace{.5cm} 0.9689 & 2.7077 \hspace{.5cm} 0.9693 & 2.8535 \hspace{.5cm} 1.0165\\
        $  1 -1/8$ & 2.4523 \hspace{.5cm} 0.9867 & 2.6903 \hspace{.5cm} 0.9864 & 2.8416 \hspace{.5cm} 1.0089\\
       $  1 -1/16$ & 2.4434 \hspace{.5cm} 0.9940 & 2.6815 \hspace{.5cm} 0.9937 & 2.8356 \hspace{.5cm} 1.0026\\
      $  1 -1/32 $ & 2.4388 \hspace{.5cm} 0.9972 & 2.6770 \hspace{.5cm} 0.9970 & 2.8327 \hspace{.5cm} 1.0078\\
      $  1 -1/64 $ & 2.4366 \hspace{.5cm} 0.9987 & 2.6748 \hspace{.5cm} 0.9984 & 2.8312 \hspace{.5cm} 0.9847\\
      $  1 -1/128$ & 2.4354 \hspace{.5cm} 0.9993 & 2.6737 \hspace{.5cm} 0.9988 & 2.8305 \hspace{.5cm} 1.0107\\
      $  1 -1/256$ & 2.4349 \hspace{.5cm} 0.9998 & 2.6731 \hspace{.5cm} 0.9994 & 2.8301 \hspace{.5cm} 1.0436\\
      $  1 -1/512$ & 2.4346 \hspace{.5cm} 1.0021 & 2.6729 \hspace{.5cm} 0.9997 & 2.8299 \hspace{.5cm} 0.8813\\
      $ 1 -1/1024$ & 2.4344 \hspace{.5cm} 1.0015 & 2.6727 \hspace{.5cm} 0.9968 & 2.8298 \hspace{.5cm} 1.0722
    \end{tabular}
  \end{center}
  \caption{Convergence of the transmission eigenvalues {for the ellipse,} as $\lambda\longrightarrow 1^-$ for  $n=4$ and $\eta=1$. Here the limiting values are $k_1(1)=2.4343$, $k_2(1)=2.6726$ and $k_3(1)=2.8300$.}
\label{n1ellipse}
\end{table} 
As we can see, we obtain the linear convergence for $\lambda\longrightarrow 1^-$ for the given ellipse as expected.
Interestingly, we also obtain linear convergence for $\lambda \longrightarrow 1^+$ for $\eta=1$ and $n=4$ although theoretically not justified. Refer to Table \ref{n2ellipse}.
\begin{table}[ht]
  \begin{center}
    \begin{tabular}{l|c |c |r r} 
      $\lambda$ & $k_1(\lambda)$ \hspace{.5cm} EOC& $k_2(\lambda)$ \hspace{.5cm} EOC&$k_3(\lambda)$ \hspace{.5cm} EOC \\
      \hline
        $ 1 +1/2$ & 2.3601 \hspace{.8cm} {N/A} & 2.5995 \hspace{.8cm}{N/A} & 2.7844 \hspace{.8cm} {N/A} \\
         $ 1 +1/4$ & 2.3974 \hspace{.5cm} 1.0101 & 2.6364 \hspace{.5cm} 1.0138 & 2.8067 \hspace{.5cm} 0.9758\\
        $  1 +1/8$ & 2.4160 \hspace{.5cm} 1.0080 & 2.6546 \hspace{.5cm} 1.0093 & 2.8181 \hspace{.5cm} 0.9880\\
       $  1 +1/16$ & 2.4252 \hspace{.5cm} 1.0047 & 2.6636 \hspace{.5cm} 1.0052 & 2.8239 \hspace{.5cm} 0.9980\\
      $  1 +1/32 $ & 2.4297 \hspace{.5cm} 1.0025 & 2.6681 \hspace{.5cm} 1.0028 & 2.8268 \hspace{.5cm} 0.9932\\
      $  1 +1/64 $ & 2.4320 \hspace{.5cm} 1.0012 & 2.6703 \hspace{.5cm} 1.0014 & 2.8283 \hspace{.5cm} 0.9966\\
      $  1 +1/128$ & 2.4331 \hspace{.5cm} 1.0006 & 2.6715 \hspace{.5cm} 1.0009 & 2.8290 \hspace{.5cm} 1.0056\\
      $  1 +1/256$ & 2.4337 \hspace{.5cm} 1.0002 & 2.6720 \hspace{.5cm} 1.0011 & 2.8294 \hspace{.5cm} 0.9796\\
      $  1 +1/512$ & 2.4340 \hspace{.5cm} 0.9989 & 2.6723 \hspace{.5cm} 1.0004 & 2.8296 \hspace{.5cm} 0.9780\\
      $ 1 +1/1024$ & 2.4341 \hspace{.5cm} 0.9982 & 2.6724 \hspace{.5cm} 1.0094 & 2.8297 \hspace{.5cm} 1.0376
    \end{tabular}
  \end{center}
  \caption{Convergence of the transmission eigenvalues {for the ellipse,} as $\lambda\longrightarrow 1^+$ for $n=4$ and $\eta=1$. Here the limiting values are $k_1(1)=2.4343$, $k_2(1)=2.6726$, and $k_3(1)=2.8300$.}
\label{n2ellipse}
\end{table}
Again, we also observe a monotonicity of the interior transmission eigenvalues with respect to $\lambda$ although we have not shown this fact from the theoretical point of view.

Finally, we show some monotonicity results for the kite-shaped domain. We first fix $\lambda=2$ as well as $\eta=-1$ and vary the index of refraction $n$. Using $120$ collocation nodes within the boundary element collocation method and the same parameters as before for the Beyn method with $\mu=5$, $\mu=6$, $\mu=7$ as well as $\mu=8$ and $\mu=8.5$ yields the first three real-valued interior transmission eigenvalues reported in Table \ref{table9}.

\begin{table}[H]
  \begin{center}
    \begin{tabular}{c|c |c|c|c|c} 
    
     $n$ & $1/7$& $1/6$ & $1/5$ & $1/4$ & $1/3$\\
    \hline 
    $k_1(n)$ & 5.6837 & 6.0582 & 6.5231 & 7.0820 & 8.1993\\ 
    $k_2(n)$ & 6.0870 & 6.2456 & 6.5370 & 7.1497 & 8.2397\\
    $k_3(n)$ & 6.6334 & 6.8110 & 7.1306 & 7.7996 & 8.9628\\
          \end{tabular}
  \end{center}
   \caption{Values of $k_j(n)$'s when $n$ varies using $\lambda=2$ and $\eta=-1$. The first real-valued transmission eigenvalue increases monotonically with respect to the parameter $n$ as stated in Theorem \ref{mono1} item 1 for the kite-shaped domain.}
\label{table9}
\end{table}

As we can see, the first real-valued interior transmission eigenvalue is monotone with respect to the parameter $n$ as stated in Theorem \ref{mono1} item 1. Interestingly, the same seems to be true for the second and third real-valued interior transmission eigenvalue. In Table \ref{table10} we show the monotonicity behavior for fixed material parameter $\lambda=2$ and $n=1/6$ and varying $\eta$ using $\lambda=4.5$, $\lambda=5.5$ as well as $\lambda=6.5$.

\begin{table}[H]
  \begin{center}
    \begin{tabular}{c|c |c|c|c|c} 
    
     $\eta$ & $-5$& $-4$ & $-3$ & $-2$ & $-1$\\
    \hline 
    $k_1(\eta)$ & 4.5272 & 5.4110 & 5.7363 & 5.9202 & 6.0582\\ 
    $k_2(\eta)$ & 5.3892 & 5.5606 & 5.7689 & 6.0044 & 6.2456\\
    $k_3(\eta)$ & 5.9899 & 6.1585 & 6.3488 & 6.5702 & 6.8110\\
          \end{tabular}
  \end{center}
   \caption{Values of $k_j(\eta)$'s when $\eta$ varies using $\lambda=2$ and $n=1/6$. The first real-valued transmission eigenvalue increases monotonically with respect to the parameter $\eta$ as stated in Theorem \ref{mono1} item 2 for the kite-shaped domain.}
\label{table10}
\end{table}

We observe the expected monotonicity behavior for the first real-valued interior transmission eigenvalue with respect to the parameter $\eta$ as stated in Theorem \ref{mono1} item 2. Strikingly, the other interior transmission eigenvalues also show a monotonicity behavior. 

Next, we show numerical results to validate Theorem \ref{mono2}. First, we pick the material parameter $\lambda=1/2$ and $\eta=1$ and vary $n$. We use $\mu=5$, $\mu=3.5$, and $\mu=3$ as well as $\mu=2$ to obtain the results reported in Table \ref{table11}.

\begin{table}[H]
  \begin{center}
    \begin{tabular}{c|c |c|c|c|c} 
    
     $n$ & $3$& $4$ & $5$ & $6$ & $7$\\
    \hline 
    $k_1(n)$ & 4.6102 & 3.4720 & 2.8104 & 2.4169 & 2.0606\\ 
    $k_2(n)$ & 4.6988 & 3.4863 & 2.8713 & 2.4513 & 2.2158\\
    $k_3(n)$ & 5.1191 & 3.8013 & 3.0731 & 2.6823 & 2.4215\\
          \end{tabular}
  \end{center}
   \caption{Values of $k_j(n)$'s when $n$ varies using $\lambda=1/2$ and $\eta=1$. The first real-valued transmission eigenvalue decreases monotonically with respect to the parameter $n$ as stated in Theorem \ref{mono2} item 1 for the kite-shaped domain.}
\label{table11}
\end{table}

As we can see, we numerically obtain the decreasing behavior for the first real-valued interior transmission eigenvalue as stated in Theorem \ref{mono2} item 1. Interestingly, we also observe a monotonic behavior for the next two interior transmission eigenvalues as well. Now, we show numerical results for the material parameters $\lambda=1/2$ and $n=3$ for varying $\eta$. Using $\mu=5$, $\mu=4.5$, and $\mu=4$ yields the results that are reported in Table \ref{table12}.

\begin{table}[H]
  \begin{center}
    \begin{tabular}{c|c |c|c|c|c} 
    
     $\eta$ & $1/2$ &$1$& $2$ & $3$ & $4$ \\
    \hline 
    $k_1(\eta)$ & 4.7339 & 4.6102 & 4.3089 & 4.0502 & 3.8981 \\ 
    $k_2(\eta)$ & 4.7572 & 4.6988 & 4.5914 & 4.3550 & 4.0804 \\
    $k_3(\eta)$ & 5.1747 & 5.1191 & 4.9526 & 4.6436 & 4.4735 \\
          \end{tabular}
  \end{center}
   \caption{Values of $k_j(\eta)$'s when $\eta$ varies using $\lambda=1/2$ and $n=3$. The first real-valued transmission eigenvalue decreases monotonically with respect to the parameter $\eta$ as stated in Theorem \ref{mono2} item 2 for the kite-shaped domain.}
\label{table12}
\end{table}

Again, we observe the proposed monotone decreasing behavior as stated in Theorem \ref{mono2} item 2 for the first real-valued interior transmission eigenvalue for the kite-shaped domain. Strikingly, the same seems to be true for the second and third eigenvalue as well.

\section{Summary and outlook}
A transmission eigenvalue problem with two conductivity parameters is considered. Existence as well as discreteness of corresponding real-valued interior transmission eigenvalues is proven. Further, it is shown that the first real-valued interior transmission eigenvalue is monotone with respect to the two parameters $\eta$ and $n$ under certain conditions. Additionally, the linear convergence for $\lambda$ against one is shown theoretically. Next, the theory is validated by extensive numerical results for a unit disk using Bessel functions. Further, numerical results are presented for general scatterer using boundary integral equations and its discretization via boundary element collocation method. Interestingly, we can show numerically monotonicity results for cases that are {not} covered yet by the theory. The existence of complex-valued interior transmission eigenvalues is still open, but it can be shown numerically that they do exist. {A worthwhile future project is to study} the case when $\lambda$ is variable. 



\begin{thebibliography}{99}

\bibitem{spectraltev1}
\newblock J. An, 
\newblock A Legendre-Galerkin spectral approximation and estimation of the index of refraction for transmission eigenvalues,
\newblock {\it Appl. Numer. Math.}, {\bf 108}, (2016), 1132--1143.

\bibitem{spectraltev2}
\newblock J. An and J. Shen, 
\newblock Spectral approximation to a transmission eigenvalue problem and its applications to an inverse problem,
\newblock {\it Comp. $\&$ Math. with Appl.}, {\bf 69(10)}, (2015), 1132--1143.


\bibitem{axler}
\newblock S. Axler, 
\newblock \emph{``Linear Algebra Done Right''},  
\newblock 3rd edition, Springer NY, 2015.


\bibitem{TA} 
\newblock L. Audibert, L. Chesnel, and H. Haddar,
\newblock Transmission eigenvalues with artificial background for explicit material index identification 
\newblock {\it C. R. Acad. Sci. Paris, Ser.}, {\bf 356(6)},  (2018), 626--631.

\bibitem{beyn2012integral}
\newblock W.-J. Beyn,
\newblock An integral method for solving nonlinear eigenvalue problems,
\newblock {\it Linear Algebra and its Applications}, {\bf 436}, (2012), 3839--3863.

\bibitem{inf-sup} 
\newblock J. Bramble, 
\newblock A proof of the inf-sup condition for the Stokes equations on Lipschitz domains. 
\newblock {\it  Math. Models and Methods in Appl. Sci.}, {\bf 13(3)}, (2003), 361--372.

\bibitem{te-cbc}
\newblock O. Bondarenko, I. Harris, and A. Kleefeld, 
\newblock The interior transmission eigenvalue problem for an inhomogeneous media with a conductive boundary, 
\newblock {\it Applicable Analysis}, {\bf 96(1)}, (2017), 2--22.


\bibitem{fmconductbc}
\newblock O. Bondarenko and X. Liu,
\newblock The factorization method for inverse obstacle scattering with conductive boundary condition,
\newblock {\it Inverse Problems}, {\bf 29} (2013), 095021.

\bibitem{approach} 
\newblock F.Cakoni, D. Colton, 
\newblock \emph{``A Qualitative Approach to Inverse Scattering Theory''} 
\newblock {\it Springer}, Berlin (2016).



\bibitem{cavities} 
\newblock  F. Cakoni, D. Colton, and H. Haddar,
\newblock The interior transmission problem for regions with cavities,
\newblock {\it  SIAM J. Math. Anal.}, {\bf 42},  (2010), 145--162.


\bibitem{far field data} 
\newblock F. Cakoni, D. Colton, and H. Haddar,
\newblock On the determination of Dirichlet or transmission  eigenvalues from far field data, 
\newblock {\it  C. R. Acad. Sci. Paris, Ser. I}, {\bf 348},  (2010), 379--383.


\bibitem{IST of TA} 
\newblock F. Cakoni, D. Colton, and H. Haddar,
\newblock \emph{``Inverse Scattering Theory and Transmission Eigenvalues''},
\newblock {\it  CBMS Series, SIAM 88}, {\bf Philadelphia},  (2016).

\bibitem{cakonikress}
F. Cakoni and R. Kress,
\newblock A boundary integral equation method for the transmission eigenvalue problem,
\newblock {\it Applicable Analysis}, {\bf 96(1)}, (2017), 23--38.

\bibitem{C and K} 
\newblock F. Cakoni and A. Kirsch,
\newblock On the interior transmission eigenvalue problem
\newblock  { \it Int. Jour. Comp. Sci. Math.}, {\bf 3}, (2010), 142--167.

\bibitem{fem-te}
\newblock F. Cakoni, P. Monk, and J. Sun, 
\newblock Error analysis of the finite element approximation of transmission eigenvalues, 
\newblock {\it Comput. Methods Appl. Math.}, {\bf 14}, (2014) 419--427. 

\bibitem{CMS}
\newblock D. Colton, P. Monk, and J. Sun J,
\newblock Analytical and computational methods for transmission eigenvalues, 
\newblock {\it Inverse Problems}, {\bf 26}, (2010), 045011

\bibitem{electro} 
\newblock A. Cossonni\`ere and H. Haddar,
\newblock The electromagnetic interior transmission problem for regions with cavities,
\newblock {\it  SIAM J. Math. Anal.}, {\bf 43},  (2011), 1698--1715.


\bibitem{te-geo-paper1}
\newblock H. Diao, X. Cao, and H. Liu, 
\newblock On the geometric structures of transmission eigenfunctions with a conductive boundary condition and applications, 
\newblock {\it Com. in Partial Differential Equation}, {\bf 46(4)}, (2021),  630--679. 


\bibitem{evans}
\newblock L. Evans, 
\newblock \emph{``Partial Differential Equation''},  
\newblock 2nd edition, AMS Providence RI, 2010.

\bibitem{GP}
\newblock D. Gintides and N. Pallikarakis,
\newblock A computational method for the inverse transmission eigenvalue problem, 
\newblock {\it Inverse Problems}, {\bf 29}, (2013), 104010.

\bibitem{anisotropic} 
\newblock  I. Harris,
\newblock ``Non-destructive testing of anisotropic materials'',
\newblock  Ph.D. Thesis, University of Delaware, (2015).

\bibitem{two-eig-cbc}
\newblock I. Harris, 
\newblock Analysis of two transmission eigenvalue problems with a coated boundary condition,
\newblock {\it  Applicable Analysis}, {\bf 100(9)}, (2021),  1996--2019.


\bibitem{mypaper1} 
\newblock I. Harris, F. Cakoni, and J. Sun, 
\newblock{Transmission eigenvalues and non-destructive testing of anisotropic magnetic materials with voids,} 
\newblock {\it Inverse Problems}, {\bf 30}, (2014), 035016.

\bibitem{te-cbc2}
\newblock I. Harris and A. Kleefeld, 
\newblock The inverse scattering problem for a conductive boundary condition and transmission eigenvalues, 
\newblock {\it Applicable Analysis}, {\bf 99(3)}, (2020), 508--529.

\bibitem{te-cbc3}
\newblock  I. Harris and A. Kleefeld,
\newblock Analysis and computation of the transmission eigenvalues with a conductive boundary condition, 
\newblock {\it Applicable Analysis}, {\bf 101(6)}, (2022), 1880--1895.

\bibitem{electro-cbc}
\newblock Y. Hao, 
\newblock Electromagnetic interior transmission eigenvalue problem for an inhomogeneous medium with a conductive boundary,
\newblock {\it  Communications on Pure $\&$ Applied Analysis}, {\bf 19(3)} (2020), 1387--1397.

\bibitem{numerical} 
\newblock A. Kleefeld,
\newblock A numerical method to compute interior transmission eigenvalues,
\newblock {\it Inverse Problems}, {\bf  29},  (2013), 104012.


\bibitem{kleefeldhot}
A. Kleefeld,
\newblock The hot spots conjecture can be false: Some numerical examples,
\newblock {\it Advances in Computational Mathematics}, {\bf 47(6)}, (2021), 85.


\bibitem{mfs-te}
\newblock A. Kleefeld and L. Pieronek,
\newblock The method of fundamental solutions for computing acoustic interior transmission eigenvalues,
\newblock {\it Inverse Problems}, {\bf 34}, (2018), 035007.

\bibitem{kirschbook}
\newblock A. Kirsch A and N. Grinberg, 
\newblock ``The Factorization Method for Inverse Problems''.
\newblock 1st edition Oxford University Press, Oxford 2008.


\bibitem{salsa} 
\newblock S. Salsa,
\newblock \emph{``Partial Differential Equations in Action From Modelling to Theory''},
\newblock Springer Italia, Milano, (2008).

\bibitem{Schm} 
\newblock U.  Schmucker,
\newblock Interpretation of induction anomalies above nonuniform surface layers
\newblock {\it Geophysics}, {\bf 36}, (1971), 156–-165.

\bibitem{eig-FEM-book}
\newblock J. Sun and  A. Zhou,
\newblock {\it ``Finite element methods for eigenvalue problems''}, 
\newblock Chapman and Hall/CRC Publications, Boca Raton, 1st Edition, (2016).

\end{thebibliography}
\end{document}